\documentclass[11pt,reqno,twoside]{amsart}
\usepackage{amssymb,amsmath,amsthm,soul,color,paralist}
\usepackage{t1enc}
\usepackage[cp1250]{inputenc}
\usepackage{a4,indentfirst,latexsym}
\usepackage{graphics}
\usepackage{mathrsfs}
\usepackage{cite,enumitem,graphicx}
\usepackage[colorlinks=true,urlcolor=blue,
citecolor=red,linkcolor=blue,linktocpage,pdfpagelabels,
bookmarksnumbered,bookmarksopen]{hyperref}
\usepackage[english]{babel}
\usepackage[left=1.50cm,right=1.50cm,top=1.50cm,bottom=1.50cm]{geometry}
\usepackage[metapost]{mfpic}
\usepackage[colorinlistoftodos]{todonotes}
\usepackage[normalem]{ulem}

\makeatletter
\providecommand\@dotsep{5}
\def\listtodoname{List of Todos}
\def\listoftodos{\@starttoc{tdo}\listtodoname}
\makeatother

\numberwithin{equation}{section}

\newcommand{\h}{H^{s}_{\e}}
\newcommand{\R}{\mathbb{R}}
\newcommand{\2}{2^{*}_{s}}

\newcommand{\C}{\mathbb{C}}

\newcommand{\N}{\mathcal{N}}
\newcommand{\M}{\mathcal{M}}

\DeclareMathOperator{\dive}{div}
\DeclareMathOperator{\supp}{supp}
\DeclareMathOperator{\e}{\varepsilon}

\newtheorem{prop}{Proposition}[section]
\newtheorem{lem}{Lemma}[section]
\newtheorem{thm}{Theorem}[section]

\newtheorem{cor}{Corollary}[section]

\newtheorem{remark}{Remark}[section]


\date{}

\begin{document}
\title[fractional Kirchhoff equations with magnetic field and critical growth]{Multiplicity and concentration of solutions for a fractional Kirchhoff equation with magnetic field and critical growth}

\author[V. Ambrosio]{Vincenzo Ambrosio}
\address{Vincenzo Ambrosio\hfill\break\indent 
Dipartimento di Ingegneria Industriale e Scienze Matematiche \hfill\break\indent
Universit\`a Politecnica delle Marche\hfill\break\indent
Via Brecce Bianche, 12\hfill\break\indent
60131 Ancona (Italy)}
\email{ambrosio@dipmat.univpm.it}

\begin{abstract}
We investigate the existence, multiplicity and concentration of nontrivial solutions for the following fractional magnetic Kirchhoff equation with critical growth:
\begin{equation*}
\left(a\varepsilon^{2s}+b\varepsilon^{4s-3} [u]_{A/\varepsilon}^{2}\right)(-\Delta)_{A/\varepsilon}^{s}u+V(x)u=f(|u|^{2})u+|u|^{\2-2}u \quad \mbox{ in } \mathbb{R}^{3},
\end{equation*}
where $\varepsilon$ is a small positive parameter, $a, b>0$ are fixed constants, $s\in (\frac{3}{4}, 1)$, $\2=\frac{6}{3-2s}$ is the fractional critical exponent, $(-\Delta)^{s}_{A}$ is the fractional magnetic Laplacian, $A:\mathbb{R}^{3}\rightarrow \mathbb{R}^{3}$ is a smooth magnetic potential, $V:\mathbb{R}^{3}\rightarrow \mathbb{R}$ is a positive continuous potential verifying the global condition due to Rabinowitz \cite{Rab}, and $f:\mathbb{R}\rightarrow \mathbb{R}$ is a $C^{1}$ subcritical nonlinearity. 
Due to the presence of the magnetic field and the critical growth of the nonlinearity, several difficulties arise in the study of our problem and a careful analysis will be needed.
The main results presented here are established by using minimax methods, concentration compactness principle of Lions \cite{Lions}, a fractional Kato's type inequality and the Ljusternik-Schnirelmann theory of critical points. 
\end{abstract}

\maketitle

\section{Introduction}

\noindent
In this paper, we deal with the following fractional Kirchhoff equation with critical growth
\begin{equation}\label{P}
\left(a\e^{2s}+b\e^{4s-3}[u]_{A/\varepsilon}^{2}\right)(-\Delta)_{A/\varepsilon}^{s}u+V(x)u=f(|u|^{2})u+|u|^{\2-2}u \quad \mbox{ in } \mathbb{R}^{3},
\end{equation}
where $\e>0$ is a small parameter, $a$ and $b$ are positive constants, $s\in (\frac{3}{4}, 1)$, $\2=\frac{6}{3-2s}$ is the fractional critical exponent,
$$
[u]_{A/\varepsilon}^{2}:=\iint_{\R^{6}} \frac{|u(x)-e^{\imath (x-y)\cdot \frac{A}{\e}(\frac{x+y}{2})} u(y)|^{2}}{|x-y|^{3+2s}} \, dxdy,
$$
$V\in C(\R^{3}, \R)$ and $A\in C^{0, \alpha}(\R^{3}, \R)$, with $\alpha\in (0, 1]$, are the electric and magnetic potentials respectively, and $(-\Delta)^{s}_{A}$ is the associated fractional magnetic operator which, up to a normalization constant, is defined  along smooth functions $u\in C^{\infty}_{c}(\R^{3}, \C)$ as
\begin{equation}\label{operator}
(-\Delta)^{s}_{A}u(x)
:=2 \lim_{r\rightarrow 0} \int_{\R^{3}\setminus B_{r}(x)} \frac{u(x)-e^{\imath (x-y)\cdot A(\frac{x+y}{2})} u(y)}{|x-y|^{3+2s}} \,dy,
\end{equation}
where $B_{r}(x)$ denotes the ball in $\R^{3}$ centered at $x\in \R^{3}$ and of radius $r>0$.
The operator \eqref{operator} has been recently introduced in \cite{DS} and can be considered
as the fractional counterpart of the magnetic Laplacian 
$$
-\Delta_{A} u:=\left(\frac{1}{\imath}\nabla-A\right)^{2}u= -\Delta u -\frac{2}{\imath} A(x) \cdot \nabla u + |A(x)|^{2} u -\frac{1}{\imath} u \dive(A(x)).
$$ 
In this context, the curl of $A$ represents a magnetic field acting on a charged particle; see \cite{LS, RS} for a physical background, and  \cite{AFF, Cingolani, CS, EL, H1, K} for some interesting existence and multiplicity results involving $-\Delta_{A}$.

When $s=\frac{1}{2}$, the definition in \eqref{operator} goes back to the '80's, and it is related to the definition of a quantized operator corresponding to the classical relativistic Hamiltonian symbol
$$
\sqrt{(\xi-A(x))^{2}+m^{2}}+V(x), \quad (\xi, x)\in \R^{3}\times \R^{3},
$$
which is the sum of the kinetic energy term involving magnetic vector potential $A(x)$ and the potential energy term of electric scalar potential $V(x)$. 
For the sake of completeness, we emphasized that in the literature there are three kinds of quantum relativistic Hamiltonians depending on how to quantize the kinetic energy term $\sqrt{(\xi-A(x))^{2}+m^{2}}$.
As explained in \cite{I10}, these three non-local operators are in general different from each other but coincide when the vector potential $A$ is assumed to be linear, which is a very relevant physical situation in $\R^{3}$.

In absence of the magnetic field (i.e. $A=0$), the operator \eqref{operator} is consistent with the following definition of fractional Laplacian $(-\Delta)^{s}$ via singular integrals (see \cite{DPV}), namely
$$
(-\Delta)^{s}u(x)
:=2 \lim_{r\rightarrow 0} \int_{\R^{3}\setminus B_{r}(x)} \frac{u(x)-u(y)}{|x-y|^{3+2s}} \,dy.
$$ 
More in general, nonlocal operators can be viewed as the infinitesimal generators of L\'evy stable diffusion processes, and arise in a quite natural way in the description of several physical phenomena such as phase transitions, continuum mechanics, population dynamics, material science, flame propagation, plasma physics and so on. Indeed, the literature on nonlocal fractional operators and on their applications is impressive, so we refer the interested reader to \cite{DPV, MBRS} for a more detailed and exhaustive description on this subject.

When $\e=a=1$, $b=0$ and $\R^{3}$ is replaced by $\R^{N}$, then \eqref{P} boils down to the following fractional magnetic Schr\"odinger equation
\begin{equation}\label{FMSE}
(-\Delta)^{s}_{A}u+V(x)u=g(x, |u|^{2})u \quad \mbox{ in } \R^{N},
\end{equation}
for which some recent existence and multiplicity results have been established; see for instance \cite{Amjm, Acpde, AD, DS, ZSZ}, and \cite{FPV, FVe} for problems in bounded domains. 
We note that, when $A\equiv 0$, \eqref{FMSE} becomes the well-known fractional Schr\"odinger equation introduced by Laskin \cite{Laskin}
\begin{equation}\label{FSE}
(-\Delta)^{s}u+V(x)u=g(x, u^2)u \quad \mbox{ in } \R^{N},
\end{equation}
as a result of expanding the Feynman path integral, from the Brownian like to the L\'evy like quantum mechanical paths. 
Equation \eqref{FSE} has received a tremendous popularity in the last decade, and in  the literature appear several works concerning the existence, multiplicity and qualitative properties of solutions obtained via suitable variational methods, and under different assumptions on the potential $V$ and the nonlinearity $g$; see \cite{AM, A1, DDPW, DPMV, FQT, H2, Is} and the references therein. 

On the other hand, if $A\neq 0$, Mingqi et al. \cite{MPSZ} dealt with the following class of fractional Schr\"odinger-Kirchhoff equations
$$
M([u]^{2}_{A})(-\Delta)^{s}_{A/\e}+V(x)u=g(x, |u|^{2})u \quad \mbox{ in } \R^{N},
$$
where $M$ is a continuous Kirchhoff function, $V$ is a positive continuous potential such that there exists $h>0$ such that $|\{x\in B_{h}(y): V(x)\leq c\}|\rightarrow 0$ as $|y|\rightarrow \infty$ for all $c>0$, and $g$ is a continuous nonlinearity with subcritical growth. In the super-and sub-linear cases, the authors proved the existence of least energy solutions for the above problem by the mountain pass theorem \cite{AR}, combined with the Nehari method, and by the direct methods respectively. Moreover, the existence of infinitely many solutions is also investigated by the symmetric mountain pass theorem.
In \cite{LRZ}, Liang et al. used a fractional version of the concentration compactness principle and critical point theory to establish a multiplicity result for the following fractional Schr\"odinger-Kirchhoff equation with electromagnetic fields and critical nonlinearity
$$
\e^{2s}M([u]^{2}_{A/\e})(-\Delta)^{s}_{A/\e}+V(x)u=|u|^{\2-2}u+g(x, |u|^{2})u \quad \mbox{ in } \R^{N},
$$
where $M$ is a positive continuous nondecreasing function, $g$ is a subcritical nonlinearity, and $V$ is such that $\min_{\R^{3}} V=0$ and there exists $\tau>0$ such that the set $\{x\in \R^{N}: V(x)<\tau\}$ has finite Lebesgue measure.

We recall that, when $A=0$, $\e=1$ and $\R^{3}$ is replaced by $\R^{N}$, equation \eqref{P} reduces to the following fractional Kirchhoff equation
\begin{align}\label{FKE}
M\left(\int_{\R^{N}}|(-\Delta)^{\frac{s}{2}}u|^{2}dx\right)(-\Delta)^{s}u+V(x)u=g(x, u) \quad \mbox{ in } \R^{N},
\end{align}
where $M$ is a continuous Kirchhoff function whose model case is given by $M(t)=a+bt$. 
It is worth pointing out that Fiscella and Valdinoci \cite{FV} proposed for the first time a stationary fractional Kirchhoff variational model with critical growth in smooth bounded domains, which takes into account the nonlocal aspect of the tension arising from nonlocal measurements of the fractional length of the string; see the Appendix in \cite{FV} for more details.
Subsequently, many authors established several existence and multiplicity results for fractional Kirchhoff equations; see \cite{AI1, FP, LSZ, MBRS} and the references therein. For instance, Fiscella and Pucci \cite{FP} dealt with the existence and the asymptotic behavior of nontrivial solutions for a class of $p$-fractional Kirchhoff type equations in $\R^{N}$ involving critical nonlinearities.
Liu et al. \cite{LSZ} used the monotonicity trick and the profile decomposition, to obtain the existence of ground states to a fractional Kirchhoff equation with critical nonlinearity in low dimension. 
The author and Isernia in \cite{AI2} studied the existence and multiplicity via penalization method and the Ljusternik-Schnirelmann category theory for a fractional Kirchhoff equation with subcritical nonlinearities.

On the other hand, when $M(t)=a+bt$ and $s=1$, then \eqref{FKE} becomes a classical Kirchhoff equation of the type
\begin{equation}\label{SKE}
-\left(a+b\int_{\R^{3}} |\nabla u|^{2}dx \right)\Delta u+V(x)u=g(x,u) \quad \mbox{ in } \R^{3}.
\end{equation}
Problem \eqref{SKE} is related to the stationary analogue of the Kirchhoff equation
\begin{equation}\label{KE}
\rho u_{tt} - \left( \frac{p_{0}}{h}+ \frac{E}{2L}\int_{0}^{L} |u_{x}|^{2} dx \right) u_{xx} =0,
\end{equation}
which was introduced by Kirchhoff \cite{Kir} in $1883$ as an extension of the classical D'Alembert's wave equations for free vibration of elastic strings. 
The Kirchhoff's model takes into account the changes in length of the string produced by transverse vibrations. In \eqref{KE}, $u=u(x, t)$ denotes the transverse string displacement at the space coordinate $x$ and time $t$, $L$ is the length of the string, $h$ is the area of the cross section, $E$ is Young's modulus of the material, $\rho$ is the mass density, and $p_{0}$ is the initial tension. We refer to \cite{B, P} for the early classical studies dedicated to \eqref{KE}. 
Anyway, only after the pioneer work of Lions \cite{LionsK}, where a functional analysis approach was proposed to attack \eqref{KE}, problem \eqref{SKE} began to call attention of several mathematicians; see \cite{ACM, FJ, HLP, HZ, WTXZ} and the references therein.
In particular, He and Zou \cite{HZ} obtained the first existence and multiplicity result of concentrating solutions 
for small $\e>0$ of the following perturbed Kirchhoff equation
\begin{align}\label{CKE}
-\left(a\varepsilon^{2}+b\varepsilon \int_{\R^{3}}|\nabla u|^{2} dx\right)\Delta u+V(x)u=g(u) \quad \mbox{ in } \mathbb{R}^{3},
\end{align} 
assuming that $V: \R^{3}\rightarrow \R$ is a continuous potential satisfying the assumption introduced by Rabinowitz \cite{Rab}:
\begin{equation*}
V_{\infty}=\liminf_{|x|\rightarrow \infty} V(x)>\inf_{x\in \R^{N}} V(x)=V_{0}, \mbox{ where } V_{\infty}\leq \infty, \tag{V}
\end{equation*}
and $g$ is a subcritical nonlinearity. Subsequently, Wang et al. \cite{WTXZ} investigated the multiplicity and concentration phenomenon for \eqref{CKE} when $g(u)=\lambda f(u)+|u|^{4}u$, $f$ is a continuous subcritical nonlinearity and $\lambda$ is large. Figueiredo and Santos Junior \cite{FJ} proved a multiplicity result for a subcritical Kirchhoff equation via the generalized Nehari manifold method, when the potential $V$ has a local minimum. He et al. \cite{HLP} considered the existence and multiplicity of solutions to \eqref{CKE} when $g(u)=f(u)+u^{5}$, $f\in C^{1}$ is a subcritical nonlinearity which does not verifies the Ambrosetti-Rabinowitz condition \cite{AR}.

Motivated by the above works, in the present paper we focus our attention on the existence, multiplicity and concentration behavior of solutions to \eqref{P} when $\e>0$ is sufficiently small, under assumption $(V)$  on the potential $V$ and
requiring that $f:\R\rightarrow \R$ is a $C^{1}$-function satisfying the following conditions:
\begin{compactenum}[$(f_1)$]
\item $f(t)=0$ for $t\leq 0$ and $\displaystyle{\lim_{t\rightarrow 0^{+}} \frac{f(t)}{t}=0}$;
\item there exist $q, \sigma\in (4, 2^{*}_{s})$ and $\lambda>0$, with $2^{*}_{s}= \frac{6}{3-2s}$, such that 
$$
f(t)\geq \lambda t^{\frac{\sigma-2}{2}}, \mbox{ for all } t>0, \mbox{ and } \quad \lim_{t\rightarrow \infty} \frac{f(t)}{t^{\frac{q-2}{2}}}=0;
$$
\item there exists $\theta\in (4, q)$ such that $0<\frac{\theta}{2} F(t)\leq t f(t)$ for all $t>0$, where $F(t)=\int_{0}^{t} f(\tau)d\tau$;
\item $\frac{f(t)}{t}$ is increasing for all $t>0$.
\end{compactenum} 
We note that the restriction $s>3/4$ allows us to deduce that $4<\2$, so that $(f_2)$ and $(f_3)$ make sense.   
These assumptions will be fundamental to derive  the  geometric  structure  of  the energy functional associated with \eqref{P},  the  boundedness and convergence  of  the  Palais-Smale sequences.\\
Let us now define the sets $M=\{x\in \R^{3}: V(x)=V_{0}\}$ and $M_{\delta}=\{x\in \R^{3}: dist(x, M)\leq \delta\}$. In order to state precisely the main results of this work, we recall that if $Y$ is a given closed set of a topological space $X$, we denote by $cat_{X}(Y)$ the Ljusternik-Schnirelmann category of $Y$ in $X$, that is the least number of closed and contractible sets in $X$ which cover $Y$; see \cite{W} for more details.

Our main result can be summarized as follows:
\begin{thm}\label{thm1}
Assume that $(V)$ and $(f_1)$-$(f_4)$ hold. Then, for any $\delta>0$
there exists $\e_{\delta}>0$ such that, for any $\e\in (0, \e_{\delta})$, problem \eqref{P} has at least $cat_{M_{\delta}}(M)$ nontrivial solutions. 
Moreover, if $u_{\e}$ denotes one of these solutions and $x_{\e}$ is a global maximum point of $|u_{\e}|$, then we have 
$$
\lim_{\e\rightarrow 0} V(x_{\e})=V_{0},
$$	
and there exists $C>0$ such that
$$
|u_{\e}(x)|\leq \frac{C\e^{3+2s}}{C\e^{3+2s}+|x-x_{\e}|^{3+2s}} \quad \forall x\in \R^{3}.
$$
\end{thm}
The proof of Theorem \ref{thm1} relies on suitable variational arguments.
We emphasize that the main difficulties in the study of \eqref{P} lie in the combination of several features such as the appearance of the magnetic field, the lack of compactness due to the unboundedness of $\R^{3}$ and the nonlinearity with critical growth. 
Therefore, in order to overcome these obstacles, a more careful and accurate analysis will be needed.
Firstly, we can prove that the energy functional $J_{\e}$ associated with \eqref{P} has a mountain pass geometry \cite{AR}.
Then, in order to get some compactness properties for $J_{\e}$, we have to circumvent several problems. The first one  is related to the presence of the magnetic Kirchhoff term $[u]^{2}_{A/\e}(-\Delta)^{s}_{A/\e}u$.  In fact, in general, we can not deduce that $[u_{n}]^{2}_{A/\e}\rightarrow [u]^{2}_{A/\e}$ if we only know that $u_{n}\rightharpoonup u$ in the fractional magnetic Sobolev space $\h$ (see Section $2$ for its definition).
Secondly, the critical exponent makes our investigation rather tricky. Indeed, due to the lack of compactness of the Sobolev embedding $\h$ in $L^{\2}(\R^{3}, \R)$, the functional $J_{\e}$ does not satisfy the $(PS)$ condition at any level $c\in \R$. More precisely, we will see that the levels of compactness depend on the behavior of the potential $V$ at infinity and on a threshold value $c_{*}$ related to the autonomous critical Kirchhoff limit problem (with $A=0$) associated with \eqref{P}. 
Thus, we will combine in an appropriate way the diamagnetic inequality established in \cite{DS} with the concentration-compactness principle of Lions \cite{Lions} to recover some compactness property for $J_{\e}$; see Lemma \ref{PSc}.
Moreover, we use the H\"older continuity of the magnetic potential $A$ and some interesting decay properties of positive solutions to the limit Kirchhoff problem associated with \eqref{P}, to get some useful estimates which will be fundamental to obtain the existence of solutions to \eqref{P} and to implement the barycenter machinery; see Sections $5$ and $6$. Finally, as in \cite{Acpde}, we are able to deduce some $L^{\infty}$ and decay estimates on the modulus of solutions of \eqref{P} applying a Moser iteration scheme \cite{Moser} and a non trivial approximation argument inspired in some sense by the Kato inequality for $-\Delta_{A}$ (see \cite{Kato}). Roughly speaking, we will show that the modulus of solutions to \eqref{P} are sub-solutions of certain fractional Kirchhoff problems (with $A=0$) and then we take advantage of the well-known polynomial decay of solutions of \eqref{FSE} (see \cite{FQT}) to achieve the desired result; see Lemma \ref{moser}. 
To the best of our knowledge, this is the first time that minimax methods jointly with the Ljusternik-Schnirelmann theory are used to get multiple solutions for \eqref{P}. Moreover, the result presented here seems to be new also in the case $s=1$.

\smallskip
The outline of this paper is as follows. In Section $2$ we give the notations and we recall some useful lemmas for the fractional magnetic Sobolev spaces. In Section $3$ we introduce the variational framework of problem \eqref{P}. The Section $4$ is devoted to the study of Palais-Smale compactness condition. In Section $5$ we give a first existence result for \eqref{P}. In the last section, we provide a multiplicity result for \eqref{P} via the Ljusternik-Schnirelmann category theory, and we study the concentration of the maximum points.

\section{Notation and preliminaries}\label{sec2}
In this section, we fix some notations and list some useful preliminary results.

For $p\in [1, \infty]$, we will use the notation $|u|_{p}$ to indicate the $L^{p}$-norm of a function $u:\R^{3}\rightarrow \R$.

Let $s\in (0, 1)$. We denote by $D^{s,2}(\R^{3}, \R)$ the completion of $C^{\infty}_{c}(\R^{3}, \R)$ with respect to the Gagliardo seminorm
$$
[u]^{2}:=[u]^{2}_{s}=\iint_{\R^{6}} \frac{|u(x)-u(y)|^{2}}{|x-y|^{3+2s}} dxdy.
$$
In particular, $D^{s,2}(\R^{3}, \R)$ can be characterized as follows:
$$
D^{s,2}(\R^{3}, \R)=\{u\in L^{\2}(\R^{3}, \R): [u]<\infty\}.
$$
Let us define the fractional Sobolev space $H^{s}(\R^{3}, \R)$ by
$$
H^{s}(\R^{3}, \R):=\{u\in L^{2}(\R^{3}, \R): [u]<\infty\}.
$$
It is well known (see \cite{DPV}) that the embedding $H^{s}(\R^{3}, \R)\subset L^{q}(\R^{3}, \R)$ is continuous for all $q\in [2, \2)$ and locally compact for all $q\in [1, \2)$. We also denote by $S_{*}=S_{*}(s)$ the best constant of the fractional Sobolev embedding $D^{s,2}(\R^{3}, \R)$ into $L^{\2}(\R^{3}, \R)$ (see \cite{CT}),
that is
$$
S_{*}:=\inf_{u\in D^{s,2}(\R^{3}, \R)\setminus\{0\}} \frac{[u]^{2}}{|u|^{2}_{\2}}.
$$

Let $L^{2}(\R^{3}, \C)$ be the space of complex-valued functions such that $|u|_{2}^{2}=\int_{\R^{3}}|u|^{2}\, dx<\infty$ endowed with the inner product 
$\langle u, v\rangle_{L^{2}}:=\Re\int_{\R^{3}} u\bar{v}\, dx$, where the bar denotes complex conjugation.

Let us define
$$
D^{s,2}_{A}(\R^3,\C):=\left\{u\in L^{2_s^*}(\R^3,\C) : [u]^{2}_{A}<\infty\right\},
$$
where
$$
[u]^{2}_{A}:=\iint_{\R^{6}} \frac{|u(x)-e^{\imath (x-y)\cdot A(\frac{x+y}{2})} u(y)|^{2}}{|x-y|^{3+2s}} \, dxdy.
$$
Then we introduce the Hilbert space
$$
H^{s}_{\e}:=\left\{u\in D_{A_{\e}}^{s,2}(\R^3,\C): \int_{\R^{3}} V_{\e}(x) |u|^{2}\, dx <\infty\right\}
$$ 
endowed with the norm
$$
\|u\|_{\e}:=\left([u]^{2}_{A_{\e}}+\int_{\R^{3}} V_{\e}(x)|u|^{2}\, dx \right)^{1/2}.
$$
It is clear that $\h$ is a Hilbert space with scalar product
\begin{align*}
&\langle u , v \rangle_{\e}:= a\Re\iint_{\R^{6}} \frac{(u(x)-e^{\imath(x-y)\cdot A_{\e}(\frac{x+y}{2})} u(y))\overline{(v(x)-e^{\imath(x-y)\cdot A_{\e}(\frac{x+y}{2})}v(y))}}{|x-y|^{3+2s}} dx dy+\Re \int_{\R^{3}} V_{\e}(x) u \bar{v} dx
\end{align*}
for all $u, v\in \h$.
In what follows, we recall some useful results related to $\h$; see \cite{AD, DS}.
\begin{lem}\cite{AD, DS}
The space $\h$ is complete and $C_c^\infty(\R^3,\C)$ is dense in $\h$. 
\end{lem}

\begin{lem}(diamagnetic inequality)\label{DI}\cite{DS}
If $u\in H^{s}_{A}(\R^{3}, \C)$ then $|u|\in H^{s}(\R^{3}, \R)$ and we have
$$
[|u|]\leq [u]_{A}.
$$
\end{lem}

\begin{thm}\label{Sembedding}\cite{DS}
	The space $H^{s}_{\e}$ is continuously embedded in $L^{r}(\R^{3}, \C)$ for all $r\in [2, 2^{*}_{s}]$, and compactly embedded in $L^{r}_{loc}(\R^{N}, \C)$ for all $r\in [1, \2)$.
	Moreover, if $V_\infty=\infty$, then, for any bounded sequence $(u_{n})$  in $\h$, we have that, up to a subsequence, $(|u_{n}|)$ is strongly convergent in $L^{r}(\R^{3}, \R)$ for all $r\in [2, 2^{*}_{s})$.
\end{thm}

\begin{lem}\label{aux}\cite{AD}
If $u\in H^{s}(\R^{3}, \R)$ and $u$ has compact support, then $w=e^{\imath A(0)\cdot x} u \in \h$.
\end{lem}

\begin{lem}\label{cutoff}\cite{DS}
Let $u\in H^{s}_{A}$ and $\varphi\in C^{0,1}(\R^{3})$ with $0\leq \varphi\leq 1$. Then, for every pair of measurable sets $E_{1}, E_{2}\subset \R^{3}$, we have
\begin{align*}
\iint_{E_{1}\times E_{2}} \frac{|u(x)\varphi(x)-u(y)\varphi(y)|^{2}}{|x-y|^{3+2s}} dx dy&\leq C\min\left\{\int_{E_1}|u|^{2}dx, \int_{E_2}|u|^{2}dx\right\} \\
&\quad +C\iint_{E_{1}\times E_{2}} \frac{|u(x)-u(y)|^{2}}{|x-y|^{3+2s}} dx dy,
\end{align*}
where $C$ depends on $s$ and on the Lipschitz constant of $\varphi$.
\end{lem}

\noindent
We also have the following fractional version of vanishing Lions lemma:
\begin{lem}\label{Lions}\cite{FQT}
	Let $r\in [2, \2)$. If $(u_{n})$ is a bounded sequence in $H^{s}(\R^{3}, \R)$ and if 
	\begin{equation*}
	\lim_{n\rightarrow \infty} \sup_{y\in \R^{3}} \int_{B_{R}(y)} |u_{n}|^{r} dx=0
	\end{equation*}
	for some $R>0$, then $u_{n}\rightarrow 0$ in $L^{t}(\R^{3}, \R)$ for all $t\in (2, 2^{*}_{s})$.
\end{lem}

\section{The variational framework}
Hereafter, we will work with the following problem equivalent to \eqref{Pe}, which is obtained via the change of variable
$u(x)\mapsto u(\e x)$:
\begin{equation}\label{Pe}
(a+b[u]^{2}_{A_{\e}})(-\Delta)^{s}_{A_{\e}} u + V_{\e}(x)u=  f(|u|^{2})u+|u|^{\2-2}u \quad \mbox{ in } \R^{3},
\end{equation}
where $A_{\e}(x):=A(\e x)$ and $V_{\e}(x):=V(\e x)$. \\
In order to find weak solutions to \eqref{Pe}, we will look for critical points of the following energy functional $J_{\e}: \h\rightarrow \R$ defined as
\begin{align*}
J_{\e}(u)=\frac{1}{2}\|u\|_{\e}^{2}+\frac{b}{4}[u]_{A_{\e}}^{4}-\frac{1}{2}\int_{\R^{3}} F(|u|^{2})dx-\frac{1}{\2}|u|_{\2}^{\2}.
\end{align*}
In view of assumptions $(f_1)$-$(f_2)$ and $(V)$, it is easy to check that $J_{\e}\in C^{1}(\h, \R)$ and that its differential $J_{\e}'$ is given by
\begin{align*}
\langle J_{\e}'(u), v\rangle &=\langle u, v\rangle_{\e}+ \Re\left(b[u]^{2}_{A_{\e}} \iint_{\R^{6}} \frac{(u(x)-e^{\imath(x-y)\cdot A_{\e}(\frac{x+y}{2})} u(y))\overline{(v(x)-e^{\imath(x-y)\cdot A_{\e}(\frac{x+y}{2})}v(y))}}{|x-y|^{3+2s}} dx dy\right) \\
&-\Re\int_{\R^{3}} f(|u|^{2})u\bar{v}+|u|^{\2-2}u\bar{v} \, dx \quad \forall u, v\in \h.
\end{align*}

\noindent
Let us observe that $J_{\e}$ has a mountain pass structure \cite{AR}, that is:
\begin{lem}\label{MPG}
\begin{compactenum}[$(i)$]
\item $J_{\e}(0)=0$;
\item there exist $\alpha, \rho>0$ such that $J_{\e}(u)\geq \alpha$ for any $u\in \h$ such that $\|u\|_{\e}=\rho$;
\item there exists $e\in \h$ with $\|e\|_{\e}>\rho$ such that $J_{\e}(e)<0$.
\end{compactenum}
\end{lem}
\begin{proof}
$(i)$ Since $F(0)=0$, it is obvious that $J_{\e}(0)=0$.\\
$(ii)$ By $(f_1)$ and $(f_2)$, for all $\delta>0$ there exists $C_{\delta}>0$ such that 
$$
|F(t^2)|\leq \delta |t|^{2}+C_{\delta}|t|^{q} \quad \mbox{ for all } x\in \R^{3}, t\in \R.
$$
This fact combined with Theorem \ref{Sembedding} implies that for $\delta>0$ small enough
$$
J_{\e}(u)\geq C\|u\|^{2}_{\e}-CC_{\delta}\|u\|_{\e}^{q}-C \|u\|^{\2}_{\e}
$$
from which the thesis follows.\\
$(iii)$ Fix $u\in \h\setminus\{0\}$ with $supp(u)\subset \Omega$, for some bounded domain $\Omega\subset \R^{3}$. Then, in view of $(f_3)$, we can see that
\begin{align*}
J_{\e}(Tu)&\leq \frac{T^{2}}{2} \|u\|^{2}_{\e}+b\frac{T^{4}}{4}[u]_{A_{\e}}^{4}-CT^{\theta}\int_{\Omega}|u|^{\theta}dx+C|\Omega|\rightarrow -\infty \quad \mbox{ as }  T\rightarrow \infty.
\end{align*}
\end{proof}

\noindent
Let us define the minimax level 
\begin{align*}
c_{\e}:=\inf_{\gamma\in \Gamma_{\e}} \max_{t\in [0, 1]} J_{\e}(\gamma(t)) \quad \mbox{ where } \quad \Gamma_{\e}:=\{\gamma\in C([0, 1], \h): \gamma(0)=0 \mbox{ and } J_{\e}(\gamma(1))<0\}.
\end{align*}
Using a version of the mountain pass theorem without $(PS)$ condition (see \cite{W}), we can find a Palais-Smale sequence $(u_{n})$ at the level $c_{\e}$. \\
Since we are looking for multiple critical points of the functional $J_{\e}$, we shall consider it constrained to an appropriated subset of $\h$.
More precisely, we introduce the Nehari manifold associated with (\ref{Pe}), that is
\begin{equation*}
\mathcal{N}_{\e}:= \{u\in \h \setminus \{0\} : \langle J_{\e}'(u), u \rangle =0\}.
\end{equation*}
We note that, for any $u\in \N_{\e}$, from $(f_{1})$-$(f_{2})$ we have for $\delta>0$ small enough
\begin{align*}
0&=\|u\|_{\e}^{2}+ b[u]^{4}_{A_{\e}}-|u|_{\2}^{\2} -\int_{\R^{3}} f(|u|^{2})|u|^{2}\, dx \\
&\geq \|u\|_{\e}^{2}- \delta |u|_{2}^{2}-C_{\delta} |u|_{q}^{q}- |u|_{\2}^{\2}\\
&\geq C\|u\|_{\e}^{2} -CC_{\delta} \|u\|_{\e}^{q} -C\|u\|_{\e}^{\2},
\end{align*}
so we can find $r>0$ such that
\begin{align}\label{uNr}
\|u\|_{\e}\geq r>0.
\end{align}
Now, we prove the following useful results:
\begin{lem}\label{lem2.2HZ}
For every $u\in \h\setminus \{0\}$, there exists a unique $t_{u}>0$ such that $t_{u}u\in \N_{\e}$. Moreover, $J_{\e}(t_{u}u)= \max_{t\geq 0} J_{\e}(tu)$. 
\end{lem}
\begin{proof}
Fix $u\in \h\setminus \{0\}$ and define $h(t):= J_{\e}(tu)$ for $t\geq 0$. Let us note that $h'(t)= \langle J'_{\e}(tu), u\rangle=0$ if and only if $tu\in \N_{\e}$. Hence $h'(t)=0$ is equivalent to
\begin{align}\label{2.3HZ}
b[u]^{4}_{A_{\e}} = - \frac{\|u\|^{2}_{\e}}{t^{2}} + \int_{\R^{3}} \frac{f(|tu|^{2})}{|tu|^{2}} |u|^{4}\, dx + t^{\2 -4} |u|_{\2}^{\2}. 
\end{align}
Using $(f_{4})$ we can see that the right hand side of \eqref{2.3HZ} is an increasing function of $t$. On the other hand, arguing as in Lemma \ref{MPG}, we can see that $h(0)=0$, $h(t)>0$ for $t>0$ small and $h(t)<0$ for $t$ large. Therefore there exists a unique $t_{u}>0$ such that $h'(t_{u})=0$, that is $t_{u}u\in \N_{\e}$. In particular, $J_{\e}(t_{u}u)= \max_{t\geq 0} J_{\e}(tu)$.  
\end{proof}

\begin{lem}\label{lemTA}
Assume that $(V)$ and $(f_{1})$-$(f_{4})$ hold true. Let $(u_{n})\subset \h$ be such that 
\begin{align*}
\langle J'_{\e}(u_{n}), u_{n}\rangle \rightarrow 0 \quad \mbox{ and } \quad \int_{\R^{3}} f(|u_{n}|^{2})|u_{n}|^{2}+ |u_{n}|^{2^{*}_{s}} \, dx \rightarrow A
\end{align*}
as $n\rightarrow \infty$, for some $A>0$. Then, up to a subsequence, there exists $t_{n}>0$ such that 
\begin{align*}
\langle J'_{\e}(t_{n}u_{n}), t_{n}u_{n}\rangle =0 \quad \mbox{ and } \quad t_{n}\rightarrow 1 \mbox{ as } n\rightarrow \infty. 
\end{align*}
\end{lem}

\begin{proof}
By the assumptions of lemma, we can deduce that $\|u_{n}\|_{\e}\neq 0$ for $n$ large. Hence, by Lemma \ref{lem2.2HZ}, we can find $t_{n}>0$ such that $t_{n}u_{n}\in \N_{\e}$, i.e. $\langle J'_{\e}(t_{n}u_{n}), t_{n}u_{n}\rangle =0$. Now, we prove that $t_{n}\rightarrow 1$ as $n\rightarrow \infty$. 
Let us define
\begin{align*}
&A_{n}:= \|u_{n}\|_{\e}^{2}, \quad B_{n}:= b[u_{n}]^{4}_{A_{\e}}, \\
&C_{n}:= \int_{\R^{3}} f(|u_{n}|^{2})|u_{n}|^{2}\, dx, \quad D_{n}:= |u_{n}|^{\2}_{\2}. 
\end{align*}
By the assumptions, we have that $C_{n}+D_{n}\rightarrow A>0$ as $n\rightarrow \infty$. Up to a subsequence, we may assume that as $n\rightarrow \infty$
\begin{align*}
A_{n}\rightarrow A_{0}, \quad B_{n}\rightarrow B_{0}, \quad C_{n}\rightarrow C_{0},\quad D_{n}\rightarrow D_{0}. 
\end{align*}
Recalling that $\langle J'_{\e}(u_{n}), u_{n}\rangle=o_{n}(1)$, we can deduce that $A_{0}+B_{0}= C_{0}+D_{0}=A$ and $A_{0}>0$ (otherwise, we obtain a contradiction in view of $A>0$). Since $\langle J'_{\e}(t_{n}u_{n}), t_{n}u_{n}\rangle =0$, by $(f_{1})$-$(f_2)$, we obtain 
\begin{align*}
t_{n}^{2}A_{n} + t_{n}^{4}B_{n} \geq t_{n}^{\sigma}C|u_{n}|^{\sigma}_{q}+ t_{n}^{2^{*}_{s}} D_{n},
\end{align*}
and 
\begin{align*}
t_{n}^{2}A_{n}+ t_{n}^{4}B_{n}\leq C\e t_{n}^{2}|u_{n}|_{2}^{2} + Ct_{n}^{q}|u_{n}|^{q}_{q} + t_{n}^{2^{*}_{s}}D_{n},
\end{align*}
which imply that there exist $T_{1}, T_{2}>0$ such that $0<T_{1}\leq t_{n}\leq T_{2}$. Hence, up to a subsequence, still denoted by $(t_{n})$, we may assume that $t_{n}\rightarrow T>0$ as $n\rightarrow \infty$. Combining $\langle J'_{\e}(t_{n}u_{n}), t_{n}u_{n}\rangle =0$ and $\langle J'_{\e}(u_{n}), u_{n}\rangle=o_{n}(1)$, we have 
\begin{align*}
t_{n}^{2}A_{n}+ t_{n}^{4}B_{n}= \int_{\R^{3}} f(|t_{n}u_{n}|^{2}) |t_{n}u_{n}|^{2}\, dx + t_{n}^{\2}D_{n}+ o_{n}(1), \quad A_{n}+B_{n}= C_{n}+D_{n}+o_{n}(1)
\end{align*}
which leads to 
\begin{align*}
\int_{\R^{3}} \left(\frac{f(|t_{n}u_{n}|^{2})}{|t_{n}u_{n}|^{2}}- \frac{f(|u_{n}|^{2})}{|u_{n}|^{2}}\right)|u_{n}|^{4}\, dx &= \left(\frac{1}{t_{n}^{2}}-1\right) A_{n}+ (1-t_{n}^{2^{*}_{s}-4})D_{n}+o_{n}(1)\\
&=\left(\frac{1}{T^{2}}-1\right) A_{0}+ (1-T^{2^{*}_{s}-4})D_{0}. 
\end{align*}
Finally, in view of $(f_4)$, we can show that $T=1$. Indeed, if $T>1$, using Fatou's Lemma, $\left(\frac{1}{T^{2}}-1\right) A_{0}+ (1-T^{2^{*}_{s}-4})D_{0}<0$, $|u_{n}|\rightarrow |u|$ a.e. in $\R^{3}$ and $(f_4)$ we deduce that 
\begin{align*}
0< \int_{\R^{3}} \left(\frac{f(|Tu|^{2})}{|Tu|^{2}}- \frac{f(|u|^{2})}{|u|^{2}}\right)|u|^{4}\, dx< 0
\end{align*}
which is a contradiction. A similar argument works when $T<1$. Consequently, $T=1$. 
\end{proof}

\noindent
Next, we define the numbers 
\begin{align*}
c_{\e}^{*}:=\inf_{u\in \N_{\e}} J_{\e}(u) \quad \mbox{ and } \quad c_{\e}^{**}:= \inf_{u\in \h\setminus\{0\}} \max_{t\geq 0} J_{\e}(tu). 
\end{align*}
In the spirit of Rabinowitz \cite{Rab}, we can prove the following useful result. 

\begin{lem}
For any fixed $\e>0$ we have $c_{\e}=c_{\e}^{*}= c_{\e}^{**}$. 
\end{lem}

\begin{proof}
From Lemma \ref{lem2.2HZ} it follows that $c_{\e}^{*}=c_{\e}^{**}$. Let us observe that for any $u\in \h\setminus \{0\}$ there exists some $t_{0}>0$ large enough such that $J_{\e}(t_{0}u)<0$. Let us define a path $\gamma: [0, 1]\rightarrow \h$ by setting $\gamma(t):= t\, t_{0}u$. It is clear that $\gamma \in \Gamma_{\e}$ and consequently $c_{\e}\leq c_{\e}^{**}$. 
In what follows, we prove that $c_{\e}\geq c_{\e}^{*}$. Note that the manifold $\N_{\e}$ separates $\h$ into two components. By $(f_{1})$ and $(f_{2})$, the component containing the origin also contains a small ball around the origin. Moreover, $J_{\e}(u)\geq 0$ for all $u$ in this component. Indeed, for all $t\in (0, t_{u}]$ we have 
\begin{align*}
h'(t)&= \langle J'_{\e}(tu), u\rangle \\
&= t^{3} \left[ b[u]^{4}_{A_{\e}} - \left(- \frac{\|u\|^{2}_{\e}}{t^{2}} + \int_{\R^{3}} \frac{f(|tu|^{2})}{|tu|^{2}} |u|^{4}\, dx + t^{\2 -4} |u|_{\2}^{\2} \right)\right] \\
&\geq \frac{t^{3}}{t_{u}^{4}} \langle J'_{\e}(t_{u}u), t_{u}u\rangle =0. 
\end{align*}
Thus, every path $\gamma \in \Gamma_{\e}$ has to cross $\N_{\e}$ and $c_{\e}^{*}\leq c_{\e}$. This ends the proof of lemma. 
\end{proof}

\begin{lem}\label{lem2.3HZ}
Let $c\in \R$. Then, any $(PS)_{c}$ sequence of $J_{\e}$ is bounded in $\h$.
\end{lem}
\begin{proof}
Let $(u_{n})\subset \h$ be a $(PS)_{c}$ sequence of $J_{\e}$, that is 
\begin{align*}
c+o_{n}(1)=\frac{1}{2}\|u_{n}\|_{\e}^{2}+\frac{b}{4}[u_{n}]^{4}_{A_{\e}}-\frac{1}{2}\int_{\R^{3}} F(|u_{n}|^{2}) dx-\frac{1}{\2}|u_{n}|_{\2}^{\2}
\end{align*}
and
\begin{align*}
o_{n}(1)=\|u_{n}\|_{\e}^{2}+b[u_{n}]^{4}_{A_{\e}}-\int_{\R^{3}} f(|u_{n}|^{2})|u_{n}|^{2} dx-|u_{n}|_{\2}^{\2}.
\end{align*}
Combining the above equalities and using $(f_3)$ we get
\begin{align*}
c+o_{n}(1)\|u_{n}\|_{\e}&= J_{\e}(u_{n})-\frac{1}{\theta}\langle J'_{\e}(u_{n}), u_{n}\rangle \\
&= \left(\frac{1}{2}-\frac{1}{\theta}\right)\|u_{n}\|^{2}_{\e}+\left( \frac{1}{4}-\frac{1}{\theta}\right)b[u_{n}]_{A_{\e}}^{4}\\
&+\frac{1}{\theta}\int_{\R^{3}} \left[f(|u_{n}|^{2})|u_{n}|^{2}-\frac{\theta}{2} F(|u_{n}|^{2})\right]\, dx\\
&+\left(\frac{1}{\theta}-\frac{1}{\2} \right) |u_{n}|_{\2}^{\2} \\
&\geq \left(\frac{1}{2}-\frac{1}{\theta}\right)\|u_{n}\|^{2}_{\e},
\end{align*}
which implies that $(u_{n})$ is bounded in $\h$.
\end{proof}

In order to obtain some compactness property for $J_{\e}$, it will be fundamental to compare $c_{\e}$ with the minimax level of the following family of autonomous problems related to \eqref{Pe}, that is for all $\mu>0$
\begin{equation}\label{AP0}
(a+b[u]^{2})(-\Delta)^{s} u + \mu u=  f(u^{2})u+|u|^{\2-2}u \quad \mbox{ in } \R^{3}.
\end{equation}
We denote by $E_{\mu}: H^{s}_{\mu}\rightarrow \R$ the corresponding energy functional defined as
\begin{align*}
E_{\mu}(u)
=\frac{1}{2}\|u\|^{2}_{\mu}+\frac{b}{4}[u]^{4}-\frac{1}{2}\int_{\R^{3}} F(u^{2})\, dx-\frac{1}{\2}|u|_{\2}^{\2},
\end{align*}
where $H^{s}_{\mu}$ stands for the fractional Sobolev space $H^{s}(\R^{3}, \R)$ endowed with the norm 
$$\|u\|_{\mu}^{2}:=a[u]^{2}+\mu |u|^{2}_{2}.$$ 
As before, we consider the Nehari manifold associated with (\ref{AP0}), that is
\begin{equation*}
\mathcal{M}_{\mu}:= \{u\in H^{s}_{\mu} \setminus \{0\} : \langle E_{\mu}'(u), u \rangle =0\},
\end{equation*}
and define $m_{\mu}:=\inf_{u\in \M_{\mu}} E_{\mu}(u)$.
From the above arguments, we can see that the number $m_{\mu}$ and the manifold $\M_{\mu}$ have properties similar to those of $c_{\e}$ and $\N_{\e}$. In particular 
\begin{align}\label{CMU}
m_{\mu}= \inf_{u\in H^{s}_{\mu}\setminus\{0\}} \max_{t\geq 0} E_{\mu}(tu)=\inf_{\gamma\in \Gamma_{\mu}} \max_{t\in [0, 1]} E_{\mu}(\gamma(t)),
%m_{\mu}:=\inf_{u\in \M_{\mu}} E_{\mu}(u)
\end{align}
where $\Gamma_{\mu}:=\{\gamma \in C([0, 1], H^{s}_{\mu}) : \gamma(0)=0, E_{\mu}(\gamma(1))<0\}$.

Now we compare the minimax level $m_{\mu}$ with a suitable constant which involves the best constant $S_{*}$ of the embedding $D^{s,2}(\R^{3}, \R)$ into $L^{\2}(\R^{3}, \R)$ (see \cite{CT}).
\begin{lem}\label{lem2.6HZ}
Let $\mu>0$. Then there exists $T>0$ such that
$$
m_{\mu}<\frac{a}{2}S_{*}T^{3-2s}+\frac{b}{4}S_{*}^{2}T^{6-4s}-\frac{1}{2^{*}_{s}}T^{3}=:c_{*}.
$$
\end{lem}
\begin{proof}
We follow some ideas used in \cite{LSZ}. Let $\eta\in C^{\infty}_{c}(\R^{3}, \R)$ be a cut-off function such that $\eta=1$ in $B_{1}(0)$, $\supp(\eta)\subset B_{2}(0)$ and $0\leq \eta\leq 1$.  We know that $S_{*}$ is achieved by $U(x)=\kappa(\nu^{2}+|x-x_{0}|^{2})^{-\frac{3-2s}{2}}$, with $\kappa\in \R$, $\nu>0$ and $x_{0}\in \R^{3}$ (see \cite{CT}). Taking $x_{0}=0$, as in \cite{SV2}, we can define $v_{\e}(x):=\eta(x)u_{\e}(x)$, $u_{\e}(x):=\e^{-\frac{3-2s}{2}}u^{*}(x/\e)$, $u^{*}(x):=\frac{U(x/S_{*}^{\frac{1}{2s}})}{|U|_{2^{*}_{s}}}$.
Then $(-\Delta)^{s}u_{\e}=|u|^{2^{*}_{s}-2}u$ in $\R^{3}$ and $[u_{\e}]^{2}=|u_{\e}|^{2^{*}_{s}}_{2^{*}_{s}}=S_{*}^{\frac{3}{2s}}$. We recall the following well-known estimates (see \cite{MBRS, SV2}):
\begin{align}
&A_{\e}:=[v_{\e}]^{2}=S^{\frac{3}{2s}}_{*}+O(\e^{3-2s}) \label{sq1}\\
&B_{\e}:=|v_{\e}|_{2}^{2}=O(\e^{3-2s}) \label{sq2}\\
&C_{\e}:=|v_{\e}|_{\sigma}^{\sigma} \geq 
\left\{
\begin{array}{ll}
O(\e^{3-\frac{(3-2s)\sigma}{2}}) &\mbox{ if } \sigma>\frac{3}{3-2s}\\
O(\log (\frac{1}{\e})\e^{3-\frac{(3-2s)\sigma}{2}}) &\mbox{ if } \sigma=\frac{3}{3-2s} \\
O(\e^{\frac{(3-2s)\sigma}{2}}) &\mbox{ if } \sigma<\frac{3}{3-2s} 
\end{array}
\right.\label{sq3}\\
&D_{\e}:=|v_{\e}|_{2^{*}_{s}}^{2^{*}_{s}}=S^{\frac{3}{2s}}_{*}+O(\e^{3}). \label{sq4}
\end{align}
Let us note that for all $\e>0$ there exists $t_{0}>0$ such that $E_{\mu}(\gamma_{\e}(t_{0}))<0$, where $\gamma_{\e}(t)=v_{\e}(\cdot/t)$.
Indeed, by $(f_2)$, we have
\begin{align}\label{sq5}
E_{\mu}(\gamma_{\e}(t))&\leq \frac{a}{2}t^{3-3s}[v_{\e}]^{2}+\frac{\mu}{2} t^{3} |v_{\e}|_{2}^{2}+\frac{b}{4}t^{6-4s}[v_{\e}]^{4}-\frac{t^{3}}{2^{*}_{s}} |v_{\e}|^{2^{*}_{s}}_{2^{*}_{s}}-\lambda\frac{t^{3}}{\sigma}|v_{\e}|_{\sigma}^{\sigma} \nonumber \\
&=\frac{a}{2}t^{3-2s}A_{\e}+\frac{b}{4}A_{\e}^{2}t^{6-4s}+\left(\mu\frac{B_{\e}}{2}-\frac{D_{\e}}{2^{*}_{s}}-\frac{\lambda C_{\e}}{\sigma}\right) t^{3}.
\end{align}
Since $0<6-4s<3$, we can use \eqref{sq2} to deduce that
$$
\mu\frac{B_{\e}}{2}-\frac{D_{\e}}{2^{*}_{s}}\rightarrow -\frac{1}{2^{*}_{s}} S_{*}^{\frac{3}{2s}}
$$
as $\e\rightarrow 0$. Hence, in view of \eqref{sq1}, we can see that for all $\e>0$ sufficiently small, $E_{\mu}(\gamma_{\e}(t))\rightarrow -\infty$ as $t\rightarrow \infty$, and then there exists $t_{0}>0$ such that $E_{\mu}(\gamma_{\e}(t_{0}))<0$. \\
Note that, as $t\rightarrow 0^{+}$ we have 
\begin{align*}
[\gamma_{\e}(t)]^{2} + |\gamma_{\e}(t)|_{2}^{2}  = t^{3-2s} A_{\e} + t^{3}B_{\e}\rightarrow 0 \mbox{ uniformly for } \e>0 \mbox{ small. }
\end{align*}
We set $\gamma_{\e}(0)=0$. Then, $\gamma_{\e}(t_{0}\cdot) \in \Gamma_{\mu}$
% where $\Gamma_{\mu}$ is defined as 
%$$
%\Gamma_{\mu}:=\{\gamma \in C([0, 1], H^{s}_{\mu}) : \gamma(0)=0, E_{\mu}(\gamma(1))<0\}.
%$$
and using \eqref{CMU} we can see that $m_{\mu}\leq \sup_{t\geq 0} E_{\mu}(\gamma_{\e}(t))$. \\
%Since $m_{\mu}= \inf_{\gamma\in \Gamma_{\mu}} \max_{t\in [0, 1]} E_{\mu}(\gamma(t))$ (see \eqref{CMU}) we can see that $m_{\mu}\leq \sup_{t\geq 0} E_{\mu}(\gamma_{\e}(t))$. \\
Taking into account that $m_{\mu}>0$, by \eqref{sq5}, there exists $t_{\e}>0$ such that 
\begin{align*}
\sup_{t\geq 0} E_{\mu}(\gamma_{\e}(t))= E_{\mu}(\gamma_{\e}(t_{\e})). 
\end{align*}
In the light of \eqref{sq1}, \eqref{sq3} and \eqref{sq5} we deduce that $E_{\mu}(\gamma_{\e}(t)) \rightarrow 0^{+}$ as $t\rightarrow 0^{+}$ and $E_{\mu}(\gamma_{\e}(t))\rightarrow -\infty$ as $t\rightarrow \infty$ uniformly for $\e>0$ small. 
Then we can find $t_{1}, t_{2}$ (independent of $\e>0$) verifying $0<t_{1}\leq t_{\e}\leq t_{2}$. \\
Set 
\begin{align*}
H_{\e}(t):= \frac{aA_{\e}}{2} t^{3-2s} + \frac{bA_{\e}^{2}}{4}t^{6-4s} -\frac{D_{\e}}{2^{*}_{s}} t^{3}. 
\end{align*}
Then we have 
\begin{align*}
m_{\mu}\leq \sup_{t\geq 0} H_{\e}(t) + \left(\frac{\mu B_{\e}}{2} -\frac{\lambda C_{\e}}{\sigma}\right) t_{\e}^{3}. 
\end{align*}
Using \eqref{sq3}, for any $\sigma\in (2, 2^{*}_{s})$, we have that $C_{\e}\geq O(\e^{3-\frac{(3-2s)\sigma}{2}})$ and exploiting \eqref{sq2} we can infer that
\begin{align*}
m_{\mu}\leq \sup_{t\geq0} H_{\e}(t)  + O(\e^{3-2s}) - O(\lambda \e^{3-\frac{(3-2s)\sigma}{2}}). 
\end{align*}
Since $3-2s>0$ and $3-\frac{(3-2s)\sigma}{2}>0$, we get 
\begin{align*}
\sup_{t\geq0} H_{\e}(t)  \geq \frac{m_{\mu}}{2} \quad \mbox{ uniformly for } \e>0 \mbox{ small. }
\end{align*}
Arguing as before there exist $t_{3}, t_{4}>0$ independent of $\e>0$ such that 
\begin{align*}
\sup_{t\geq0} H_{\e}(t) = \sup_{t\in [t_{3}, t_{4}]} H_{\e}(t). 
\end{align*}
By \eqref{sq1} we deduce 
\begin{align}\label{sq6}
m_{\mu}\leq \sup_{t\geq0} K(S_{*}^{\frac{1}{2s}} t) +O(\e^{3-2s}) - O(\lambda \e^{3-\frac{(3-2s)\sigma}{2}}), 
\end{align}
where 
\begin{align}\label{defK}
K(t):= \frac{aS_{*}}{2} t^{3-2s} +\frac{bS_{*}^{2}}{4} t^{6-4s} -\frac{1}{2^{*}_{s}} t^{3}. 
\end{align}
Let us note that 
\begin{align*}
K'(t)&= \frac{3-2s}{2} aS_{*} t^{2-2s} + \frac{3-2s}{2} bS_{*}^{2} t^{5-4s} - \frac{3-2s}{2} t^{2} \\
&= \frac{(3-2s)t^{2-2s}}{2} \left( aS_{*} + bS_{*}^{2} t^{3-2s} -t^{2s}\right) =: \frac{(3-2s)t^{2-2s}}{2} \tilde{K}(t). 
\end{align*}
Moreover, 
\begin{align*}
\tilde{K}'(t)= bS_{s} (3-2s)t^{2-2s} - 2s t^{2s-1}= t^{2-2s} [bS_{*}^{2}(3-2s) - 2s t^{4s-3}]. 
\end{align*}
Since $4s>3$, there exists a unique $T>0$ such that $\tilde{K}(t)>0$ for $t\in (0, T)$ and $\tilde{K}(t)<0$ for $t>T$. Thus, $T$ is the unique maximum point of $K$. In virtue of \eqref{sq6} we have 
\begin{align}\label{sq7}
m_{\mu}\leq K(T)+ O(\e^{3-2s}) - O(\lambda \e^{3-\frac{(3-2s)\sigma}{2}}). 
\end{align}
If $\sigma>\frac{4s}{3-2s}$, then $0<3-\frac{(3-2s)\sigma}{2}<3-2s$, and by \eqref{sq7}, for any fixed $\lambda>0$, we obtain $m_{\mu}<K(T)$ for $\e>0$ sufficiently small. \\
If $2<\sigma\leq\frac{4s}{3-2s}$, then, for $\e>0$ small and $\lambda>\e^{\frac{(3-2s)\sigma}{2}-2s-1}$, we also have $m_{\mu}<K(T)$. This ends the proof of lemma.
\end{proof}

\begin{remark}
Let us note that in the case $V_{\infty}<\infty$, we can deduce that $m_{V_{\infty}}<c_{*}$.
\end{remark}

\noindent
Next, we establish an existence result for autonomous problem \eqref{AP0}. More precisely, we have:
\begin{lem}\label{FS}
For all $\mu>0$, there exists a positive ground state solution of \eqref{AP0}.
\end{lem}
\begin{proof}
It is easy to check that $E_{\mu}$ has a mountain pass geometry, so, using a version of the mountain pass theorem without $(PS)$ condition (see \cite{W}), there exists a sequence $(u_{n})\subset H^{s}_{\mu}$ such that $E_{\mu}(u_{n})\rightarrow m_{\mu}$ and $E'_{\mu}(u_{n})\rightarrow 0$.  Thus, $(u_{n})$ is bounded in $H^{s}_{\mu}$ and we may assume that $u_{n}\rightharpoonup u$ in $H^{s}_{\mu}$ and $[u_{n}]^{2}\rightarrow B^{2}$. 
Suppose that $u\neq 0$. Since $\langle E'_{\mu}(u_{n}), \varphi\rangle=o_{n}(1)$ we can deduce that for all $\varphi\in C^{\infty}_{c}(\R^{3}, \R)$
\begin{align}\label{HZman}
 \int_{\R^{3}} a(-\Delta)^{\frac{s}{2}} u (-\Delta)^{\frac{s}{2}}\varphi+\mu u \varphi \, dx+& bB^{2} \left( \int_{\R^{3}} (-\Delta)^{\frac{s}{2}} u (-\Delta)^{\frac{s}{2}}\varphi \, dx\right) \nonumber \\
 &=\int_{\R^{3}} [f(u^{2})u+|u|^{\2-2}u]\varphi \, dx.
\end{align}
Let us note that $B^{2}\geq [u]^{2}$ by Fatou's Lemma. If by contradiction $B^{2}>[u]^{2}$, we may use \eqref{HZman} to deduce that $\langle E'_{\mu}(u), u\rangle<0$. Moreover, conditions $(f_1)$-$(f_2)$ imply that $\langle E'_{\mu}(t u), t u\rangle>0$ for small $t>0$. Then there exists $t_{0}\in (0, 1)$ such that $t_{0} u\in \M_{\mu}$ and $\langle E'_{\mu}(t_{0} u), t_{0} u\rangle=0$. Using Fatou's Lemma, $t_{0}\in (0, 1)$, $s\in (\frac{3}{4},1)$ and  $\frac{1}{4}f(t)t-\frac{1}{2}F(t)$ is increasing for $t>0$ (by $(f_3)$ and $(f_4)$) we get
\begin{align*}
m_{\mu}&\leq E_{\mu}(t_{0} u)-\frac{1}{4} \langle E'_{\mu}(t_{0} u), t_{0} u\rangle<\liminf_{n\rightarrow \infty} \left[E_{\mu}(u_{n})-\frac{1}{4} \langle E'_{\mu}(u_{n}), u_{n}\rangle\right]=c_{\mu}
\end{align*}
which gives a contradiction. Therefore $B^{2}= [u]^{2}$ and we deduce that $E'_{\mu}(u)=0$. Hence  $u\in \M_{\mu}$. In view of $\langle E'_{\mu}(u), u^{-}\rangle=0$ and $(f_1)$ we can see that $u\geq 0$ in $\R^{3}$. Moreover, we can argue as in Lemma $6.1$ in \cite{AI2} to infer that $u\in L^{\infty}(\R^{3}, \R)$. Since $u$ satisfies
$$
(-\Delta)^{s}u=(a+b[u]^{2})^{-1}[f(u^{2})u+|u|^{\2-2}u-\mu u]\in L^{\infty}(\R^{3}, \R),
$$
and $s>3/4$, we obtain $u\in C^{1, \gamma}(\R^{3}, \R)\cap L^{\infty}(\R^{3},\R)$, for some $\gamma<2s-1$ (see \cite{S}) and that $u>0$ by maximum principle. Now we prove that $E_{\mu}(u)=m_{\mu}$. Indeed, by $u\in \M_{\mu}$, $(f_3)$, $(f_4)$ and Fatou's Lemma we have
\begin{align*}
m_{\mu}&\leq E_{\mu}(u)-\frac{1}{4}\langle E'_{\mu}(u), u\rangle
\leq\liminf_{n\rightarrow \infty} E_{\mu}(u_{n})-\frac{1}{4}\langle E'_{\mu}(u_{n}), u_{n}\rangle 
=m_{\mu}.
\end{align*}
Now, we consider the case $u=0$. Since $m_{\mu}>0$ and $E_{\mu}$ is continuous, we can see that $\|u_{n}\|_{\mu}\nrightarrow 0$. 
Then, we claim to prove that there exists a sequence $(y_{n})\subset \R^{3}$ and constants $R, \beta>0$ such that
$$
\liminf_{n\rightarrow \infty} \int_{B_{R}(y_{n})}|u_{n}|^{2}dx\geq \beta>0.
$$
This can be done arguing as in Lemma \ref{PSc} below (see Step 1), so we omit the details.
Therefore, we can  define $v_{n}(x)=u_{n}(x+y_{n})$ such that $v_{n}\rightharpoonup v$ in $H^{s}_{\mu}$ for some $v\neq 0$. Proceeding as in the previous case we get the thesis.
\end{proof}

\noindent
Using similar arguments to those in Lemma \ref{FS}, we can prove the following fundamental compactness result for \eqref{AP0}. 
\begin{lem}\label{FSC}
Let $(u_{n})\subset \M_{\mu}$ be a sequence such that $E_{\mu}(u_{n})\rightarrow m_{\mu}$. Then, we have either
\begin{compactenum}[$(i)$]
\item $(u_{n})$ has a subsequence strongly convergent in $H^{s}_{\mu}$, or 
\item there exists $(\tilde{y}_{n})\subset \R^{3}$ such that the sequence $v_{n}=u_{n}(x+ \tilde{y}_{n})$ strongly converges in $H^{s}_{\mu}$. 
\end{compactenum}
\end{lem}

\section{The Palais-Smale condition}
In this section we prove that $J_{\e}$ verifies the Palais-Smale condition in a suitable sublevel, related to the ground energy at infinity. Due to the the presence of the critical Sobolev exponent, we will use the following well-known concentration-compactness principle due to Lions \cite{Lions}.
\begin{prop}\label{propL}
Let $(\rho_{n})$ be a sequence in $L^{1}(\R^{N})$ such that
$$
\rho_{n}\geq 0 \mbox{ in } \R^{N} \quad \mbox{ and } \quad \int_{\R^{N}} \rho_{n}(x)dx=\lambda,
$$
where $\lambda>0$ is fixed.
Then there exists a subsequence $(\rho_{n_{k}})$ satisfying one of the three following possibilities:
\begin{compactenum}[(i)]
\item (compactness) there exists $y_{k}\in \R^{N}$ such that $\rho_{n_{k}}(\cdot+y_{k})$ is tight, that is for any $\e>0$ there exists $R>0$ such that 
$$
\int_{y_{k}+B_{R}} \rho_{n_{k}}(x)dx\geq \lambda-\e;
$$
\item (vanishing) for any fixed $R>0$, there holds
$$
\lim_{k\rightarrow \infty} \sup_{y\in \R^{N}} \int_{y+B_{R}} \rho_{n_{k}}(x)dx=0;
$$
\item (dichotomy) there exists $\alpha\in (0, \ell)$ such that for all $\e>0$, there exists $k_{0}\geq 1$ and $\rho_{k}^{(1)}, \rho_{k}^{(2)}\in L^{1}(\R^{N})$, with $\rho_{k}^{(1)}, \rho_{k}^{(2)}\geq 0$, satisfying for $n\geq n_{0}$ 
\begin{align*}
|\rho_{n_{k}}-(\rho_{k}^{(1)}+\rho_{k}^{(2)})|_{1}<\e, \quad \left|\int_{\R^{N}} \rho^{(1)}_{k}(x) dx-\alpha \right|<\e, \quad \left|\int_{\R^{N}} \rho^{(2)}_{k}(x) dx-(\ell-\alpha) \right|<\e
\end{align*}
and 
$$
dist(supp(\rho^{(1)}_{k}), supp(\rho^{(2)}_{k}))\rightarrow \infty, \mbox{ as } k\rightarrow \infty. 
$$
\end{compactenum}
\end{prop}

Now, we prove the following compactness property for the unconstrained functional.
\begin{lem}\label{PSc}
$J_{\e}$ satisfies the $(PS)_{c}$ conditions at any level $0<c<m_{V_{\infty}}$ if $V_{\infty}<\infty$, and at any level $0<c<c_{*}$ if $V_{\infty}=\infty$. 
\end{lem}
\begin{proof}
Let $(u_{n})\subset \h$ be a $(PS)_{c}$ sequence of $J_{\e}$, that is 
\begin{equation}\label{3.3LG}
J_{\e}(u_{n})\rightarrow c \mbox{ and } J_{\e}'(u_{n})\rightarrow 0,
\end{equation}
as $n\rightarrow \infty$.
By Lemma \ref{lem2.3HZ} we know that $(u_{n})$ is bounded in $\h$. 
Then, we may assume that $u_{n}\rightharpoonup u$ in $\h$ and
\begin{equation}\label{B^2}
[u_{n}]^{2}_{A_{\e}}\rightarrow B^{2}.
\end{equation}
Firstly, we consider the case $V_{\infty}<\infty$.
Let us define
\begin{align*}
\rho_{n}(x)&:=\frac{a}{4}\int_{\R^{3}} \frac{|u_{n}(x)-u_{n}(y)e^{\imath (x-y)\cdot A_{\e}(\frac{x+y}{2})}|^{2}}{|x-y|^{3+2s}} dy+V_{\e}(x)|u_{n}(x)|^{2}\\
&+\frac{1}{4}\left[f(|u_{n}(x)|^{2})|u_{n}(x)|^{2}-2F(|u_{n}(x)|^{2})\right]+\frac{4s-3}{12}|u_{n}(x)|^{\2}.
\end{align*}
We note that $\rho_{n}\geq 0$ because of $(f_3)$ and $s>\frac{3}{4}$.
Moreover, $(\rho_{n})$ is bounded in $L^{1}(\R^{3})$, so we may assume, up to subsequence, that
$$
\Phi(u_{n}):=|\rho_{n}|_{1}\rightarrow \ell \quad \mbox{ as } n\rightarrow \infty.
$$
Clearly $\ell>0$, otherwise $J_{\e}(u_{n})\rightarrow 0$ which is impossible. Indeed, $\ell=c$ (since $c+o_{n}(1)=J_{\e}(u_{n})-\frac{1}{4}\langle J'_{\e}(u_{n}), u_{n}\rangle=\Phi(u_{n})$).
Now, we aim to apply Proposition \ref{propL} to $(\rho_{n})$. We will show that neither vanishing nor dichotomy occurs.

%Firstly we prove that neither vanishing nor dichotomy occurs.
\medskip
{\bf Step 1} The vanishing does not occur.\\
If $(\rho_{n})$ vanishes, then we can deduce that there exists $R>0$ such that 
$$
\lim_{n\rightarrow \infty} \sup_{y\in \R^{3}} \int_{B_{R}(y)} |u_{n}|^{2}dx=0.
$$
Therefore, in view of Lemma \ref{Lions}, we get $|u_{n}|\rightarrow 0$ in $L^{p}(\R^{3}, \R)$ for all $p\in (2, \2)$. This fact and $(f_1)$-$(f_2)$ imply that
\begin{equation}\label{Lf}
\int_{\R^{3}} f(|u_{n}|^{2})|u_{n}|^{2}dx=\int_{\R^{3}} F(|u_{n}|^{2})dx=o_{n}(1).
\end{equation}
Let us consider
\begin{align}\label{defIe}
I_{\e}(u):=J_{\e}(u)+\frac{bB^{2}}{2}[u]^{2}_{A_{\e}}-\frac{b}{4}[u]^{4}_{A_{\e}}.
\end{align}
It is easy to check that
\begin{align}\label{ieun}
I_{\e}(u_{n})=c+\frac{bB^{4}}{4}+o_{n}(1) \mbox{ and }  I'_{\e}(u_{n})=o_{n}(1).
\end{align}
Taking into account $(u_{n})$ is bounded, $\langle I'_{\e}(u_{n}), u_{n}\rangle=o_{n}(1)$ and \eqref{Lf} we can deduce that
\begin{align}\label{LLL1}
o_{n}(1)=\|u_{n}\|^{2}_{\e}+bB^{2}[u_{n}]^{2}_{\e}-|u_{n}|_{\2}^{\2}
\end{align}
which together with the Sobolev embedding $D^{s,2}(\R^{3}, \R) \subset L^{\2}(\R^{3}, \R)$, the definition of $\|\cdot\|_{\e}$ and \eqref{B^2} yields that
\begin{equation}\label{LLL2}
aS_{*}|u_{n}|_{\2}^{2}+bS_{*}^{2}|u_{n}|_{\2}^{4}\leq |u_{n}|_{\2}^{\2}+o_{n}(1).
\end{equation}
Assume that $|u_{n}|_{\2}^{\2}\rightarrow \mathcal{L}^{3}\geq 0$. Clearly, $\mathcal{L}>0$ and using \eqref{LLL2} we obtain that 
$$
aS_{*}\mathcal{L}^{3-2s}+bS_{*}^{2}\mathcal{L}^{6-4s}-\mathcal{L}^{3}\leq 0,
$$
which implies that $\mathcal{L}\geq T$, where $T$ is the unique maximum point of $K(t)$ which is defined in \eqref{defK}. Therefore, from \eqref{ieun}, \eqref{LLL1} and $\mathcal{L}\geq T$ we can see that
\begin{align}\begin{split}\label{LT}
c&=I_{\e}(u_{n})-\frac{bB^{4}}{4}-\frac{1}{\2}\langle I'_{\e}(u_{n}), u_{n}\rangle+o_{n}(1)\\
&=\left(\frac{a}{2}-\frac{a}{\2}\right)[u_{n}]_{A_{\e}}^{2}+ \left(\frac{1}{2}-\frac{1}{\2}\right) \int_{\R^{3}} V_{\e}|u_{n}|^{2}dx+\left(\frac{bB^{2}}{2}[u_{n}]^{2}_{A_{\e}}- \frac{bB^{4}}{4}\right)-\frac{bB^{2}}{\2}[u_{n}]^{2}_{A_{\e}}+o_{n}(1)\\
&\geq \left(\frac{a}{2}-\frac{a}{\2}\right)[u_{n}]_{A_{\e}}^{2}+\left(\frac{1}{2}-\frac{1}{\2}\right) \int_{\R^{3}} V_{\e}|u_{n}|^{2}dx+\left(\frac{b}{4}-\frac{b}{\2}\right)[u_{n}]^{4}_{A_{\e}}+o_{n}(1) \\
&\geq \left(\frac{a}{2}-\frac{a}{\2}\right)S_{*}|u_{n}|_{\2}^{2}+\left(\frac{b}{4}-\frac{b}{\2}\right)S^{2}_{*}|u_{n}|_{\2}^{4}+o_{n}(1) \\
&=\left(\frac{a}{2}-\frac{a}{\2}\right)S_{*}\mathcal{L}^{3-2s}+\left(\frac{b}{4}-\frac{b}{\2}\right)S^{2}_{*}\mathcal{L}^{6-4s}\\
&\geq \left(\frac{a}{2}-\frac{a}{\2}\right)S_{*}T^{3-2s}+\left(\frac{b}{4}-\frac{b}{\2}\right)S^{2}_{*}T^{6-4s} \\
&=\frac{a}{2}S_{*}T^{3-2s}+\frac{b}{4}T^{6-4s}-\frac{1}{\2}T^{3}=c_{*}
\end{split}\end{align}
which leads to a contradiction. 
%Consequently, the vanishing does not occur.

{\bf Step 2} The dichotomy does not occur. \\
Assume by contradiction that there is $\alpha\in (0, \ell)$ and $(y_{n})\subset \R^{3}$ such that for every $\eta_{n}\rightarrow 0$, we can choose $(R_{n})\subset \R_{+}$ ($R_{n}>R_{0}/\e+R'$, for any fixed $\e>0$, $R_{0}, R'$ are positive constants defined later) with $R_{n}\rightarrow \infty$ such that
\begin{align}\label{Dun}
\limsup_{n\rightarrow \infty} \left|\alpha-\int_{B_{R_{n}}(y_{n})} \rho_{n}(x)dx\right|+\left|(\ell-\alpha)-\int_{B^{c}_{R_{n}}(y_{n})} \rho_{n}(x) dx\right|<\eta_{n}.
\end{align}
Let $\xi:\R^{+}\rightarrow [0,1]$ be a cut-off function such that $\xi(t)=1$ for $t\leq 1$, $\xi(t)=0$ for $t\leq 2$ and $|\xi'(t)|\leq 2$. Let us define
$$
v_{n}(x):=\xi\left(\frac{|x-y_{n}|}{R_{n}}\right)u_{n}(x) \mbox{ and } w_{n}(x):=\left(1-\xi\left(\frac{|x-y_{n}|}{R_{n}}\right)  \right) u_{n}(x).
$$ 
In order to lighten the notation we set $\xi_{n}(x):=\xi\left(\frac{|x-y_{n}|}{R_{n}}\right)$.\\
Let $\Omega_{n}=B_{2R_{n}}(y_{n})\setminus B_{R_{n}}(y_{n})$. By \eqref{Dun} it follows that $\int_{\Omega_{n}} \rho_{n}(x) dx\rightarrow 0$ which implies that
\begin{align}\begin{split}\label{TL}
&\int_{\Omega_{n}} \left[ a\int_{\R^{3}} \frac{|u_{n}(x)-u_{n}(y)e^{\imath (x-y)\cdot A_{\e}(\frac{x+y}{2})}|^{2}}{|x-y|^{3+2s}} dy+V_{\e}(x)|u_{n}(x)|^{2}\right] dx\rightarrow 0, \\
&\int_{\Omega_{n}} V_{\e} |u_{n}|^{2} dx\rightarrow 0, \\
&\int_{\Omega_{n}}|u_{n}|^{\2} dx\rightarrow 0.
\end{split}\end{align}
Combining \eqref{TL} and Lemma \ref{cutoff} we can see that
\begin{align}\label{39T}
\iint_{\Omega_{n}\times \R^{3}} \frac{|v_{n}(x)-v_{n}(y)e^{\imath (x-y)\cdot A_{\e}(\frac{x+y}{2})}|^{2}}{|x-y|^{3+2s}} dy&\leq C\int_{\Omega_{n}}|u_{n}|^{2}dx  \nonumber\\
&\quad+C\iint_{\Omega_{n}\times \R^{3}}  \frac{|u_{n}(x)-u_{n}(y)e^{\imath (x-y)\cdot A_{\e}(\frac{x+y}{2})}|^{2}}{|x-y|^{3+2s}} dy\rightarrow 0
\end{align}
and
\begin{align}\label{40T}
\iint_{\Omega_{n}\times \R^{3}}  \frac{|w_{n}(x)-w_{n}(y)e^{\imath (x-y)\cdot A_{\e}(\frac{x+y}{2})}|^{2}}{|x-y|^{3+2s}} dy&\leq C\int_{\Omega_{n}}|u_{n}|^{2}dx  \nonumber\\
&\quad+C\iint_{\Omega_{n}\times \R^{3}}  \frac{|u_{n}(x)-u_{n}(y)e^{\imath (x-y)\cdot A_{\e}(\frac{x+y}{2})}|^{2}}{|x-y|^{3+2s}} dy\rightarrow 0.
\end{align}
Now, we show that the following relations hold true:
\begin{align}\begin{split}\label{3.9LG}
&[u_{n}]_{A_{\e}}^{2}=[v_{n}]_{A_{\e}}^{2}+[w_{n}]_{A_{\e}}^{2}+o_{n}(1), \\
&\int_{\R^{3}} V_{\e} |u_{n}|^{2}=\int_{\R^{3}} V_{\e} |v_{n}|^{2}+\int_{\R^{3}} V_{\e} |w_{n}|^{2}+o_{n}(1)\\ 
&\int_{\R^{3}} F(|u_{n}|^{2}) dx=\int_{\R^{3}} F(|v_{n}|^{2}) dx+\int_{\R^{3}} F(|w_{n}|^{2}) dx+o_{n}(1)\\
&\int_{\R^{3}} f(|u_{n}|^{2})|u_{n}|^{2} dx=\int_{\R^{3}} f(|v_{n}|^{2})|v_{n}|^{2} dx+\int_{\R^{3}} f(|w_{n}|^{2})|w_{n}|^{2} dx+o_{n}(1)\\
&|u_{n}|^{\2}_{\2}=|v_{n}|_{\2}^{\2}+|w_{n}|_{\2}^{\2}+o_{n}(1).
\end{split}\end{align}
%Taking into account \eqref{39T} and \eqref{40T} and arguing as in Lemma $3.9$ in \cite{DS} we can see that the first identity holds true.
Let us observe that being $u_{n}=v_{n}+w_{n}$ it holds
\begin{align}\label{UVW}
[u_{n}]_{A_{\e}}^{2}&=[v_{n}]_{A_{\e}}^{2}+[w_{n}]_{A_{\e}}^{2} \nonumber\\
&\quad+2\Re\iint_{\R^{6}} \frac{(v_{n}(x)-v_{n}(y)e^{\imath A_{\e}(\frac{x+y}{2})\cdot (x-y)})\overline{(w_{n}(x)-w_{n}(y)e^{\imath A_{\e}(\frac{x+y}{2})\cdot (x-y)})}}{|x-y|^{3+2s}} dx dy.
\end{align}
Set 
\begin{align}\label{Hn}
H_{n}:=\Re\iint_{\R^{6}} \frac{(v_{n}(x)-v_{n}(y)e^{\imath A_{\e}(\frac{x+y}{2})\cdot (x-y)})\overline{(w_{n}(x)-w_{n}(y)e^{\imath A_{\e}(\frac{x+y}{2})\cdot (x-y)})}}{|x-y|^{3+2s}} dx dy.
\end{align}
Then we can see that $H_{n}$ can be estimated as follows
\begin{align}\label{Hn1234}
|H_{n}|\leq \sum_{i=1}^{4}  H^{i}_{n},
\end{align}
where
\begin{align*}
&H^{1}_{n}:=2\iint_{B_{R_{n}}(y_{n})\times B^{c}_{2R_{n}}(y_{n})} \frac{|u_{n}(x)||u_{n}(y)|}{|x-y|^{3+2s}} dx dy, \\
&H_{n}^{2}:=2\iint_{\Omega_{n}\times B_{R_{n}}(y_{n})} \frac{|v_{n}(x)-v_{n}(y)e^{\imath A_{\e}(\frac{x+y}{2})\cdot (x-y)}| |w_{n}(x)-w_{n}(y)e^{\imath A_{\e}(\frac{x+y}{2})\cdot (x-y)}|}{|x-y|^{3+2s}} dx dy, \\
&H_{n}^{3}:=\iint_{\Omega_{n}\times \Omega_{n}} \frac{|v_{n}(x)-v_{n}(y)e^{\imath A_{\e}(\frac{x+y}{2})\cdot (x-y)}| |w_{n}(x)-w_{n}(y)e^{\imath A_{\e}(\frac{x+y}{2})\cdot (x-y)}|}{|x-y|^{3+2s}} dx dy, \\
&H_{n}^{4}:=2\iint_{\Omega_{n}\times B^{c}_{2R_{n}}(y_{n})} \frac{|v_{n}(x)-v_{n}(y)e^{\imath A_{\e}(\frac{x+y}{2})\cdot (x-y)}| |w_{n}(x)-w_{n}(y)e^{\imath A_{\e}(\frac{x+y}{2})\cdot (x-y)}|}{|x-y|^{3+2s}} dx dy.
\end{align*}
Now, we estimate each $H_{n}^{i}$ for $i=1, 2, 3, 4$.
Using the H\"older inequality and the fact that if $x\in B_{R_{n}}(y_{n})$ and $y\in B^{c}_{2R_{n}}(y_{n})$ then $|x-y|\geq R_{n}$, we can see that
\begin{align}\begin{split}\label{Hn1}
H_{n}^{1}&\leq \iint_{B_{R_{n}}(y_{n})\times B^{c}_{2R_{n}}(y_{n})} \frac{|u_{n}(x)|^{2}+|u_{n}(y)|^{2}}{|x-y|^{3+2s}} dx dy\\
&=\int_{B_{R_{n}}(y_{n})} |u_{n}(x)|^{2} \left(\int_{B^{c}_{2R_{n}}(y_{n})} \frac{1}{|x-y|^{3+2s}} dy\right) dx\\
&\quad+\int_{B^{c}_{2R_{n}}(y_{n})} |u_{n}(y)|^{2} \left(\int_{B_{R_{n}}(y_{n})} \frac{1}{|x-y|^{3+2s}} dx\right) dy\\
&\leq \int_{B_{R_{n}}(y_{n})} |u_{n}(x)|^{2} \left(\int_{|x-y|\geq R_{n}} \frac{1}{|x-y|^{3+2s}} dy\right) dx\\
&\quad+\int_{B^{c}_{2R_{n}}(y_{n})} |u_{n}(y)|^{2} \left(\int_{|x-y|\geq R_{n}} \frac{1}{|x-y|^{3+2s}} dx\right) dy \\
&\leq \frac{C}{R^{2s}_{n}}\rightarrow 0
\end{split}\end{align}
as $n\rightarrow \infty$. On the other hand, by Lemma \ref{cutoff}, $V_{\e}\geq V_{0}$, \eqref{TL} and the definitions of $v_{n}$ and $w_{n}$ we get
\begin{align}\begin{split}\label{Hn2}
H^{2}_{n}&\leq \iint_{\Omega_{n}\times B_{R_{n}}(y_{n})} \frac{|v_{n}(x)-v_{n}(y)e^{\imath A_{\e}(\frac{x+y}{2})\cdot (x-y)}|^{2}}{|x-y|^{3+2s}} dx dy\\
&\quad+\iint_{\Omega_{n}\times B_{R_{n}}(y_{n})} \frac{|w_{n}(x)-w_{n}(y)e^{\imath A_{\e}(\frac{x+y}{2})\cdot (x-y)}|^{2}}{|x-y|^{3+2s}} dx dy \\
&\leq C\left(\int_{\Omega_{n}} V_{\e} |u_{n}|^{2}dx+\iint_{\Omega_{n}\times B_{R_{n}}(y_{n})} \frac{|u_{n}(x)-u_{n}(y)e^{\imath A_{\e}(\frac{x+y}{2})\cdot (x-y)}|^{2}}{|x-y|^{3+2s}} dx dy \right) \\
&\leq C\left(\int_{\Omega_{n}} V_{\e} |u_{n}|^{2}dx+\iint_{\Omega_{n}\times \R^{3}} \frac{|u_{n}(x)-u_{n}(y)e^{\imath A_{\e}(\frac{x+y}{2})\cdot (x-y)}|^{2}}{|x-y|^{3+2s}} dx dy \right)\rightarrow 0.
\end{split}\end{align}
In similar fashion we can show that $H_{n}^{4}\rightarrow 0$ as $n\rightarrow \infty$.
Finally, gathering \eqref{39T} and \eqref{40T} we can see that 
\begin{align}\begin{split}\label{Hn3}
H^{3}_{n}&\leq \iint_{\Omega_{n}\times \Omega_{n}} \frac{|v_{n}(x)-v_{n}(y)e^{\imath A_{\e}(\frac{x+y}{2})\cdot (x-y)}|^{2}}{|x-y|^{3+2s}} dx dy \\
&\quad+\iint_{\Omega_{n}\times \Omega_{n}} \frac{|w_{n}(x)-w_{n}(y)e^{\imath A_{\e}(\frac{x+y}{2})\cdot (x-y)}|^{2}}{|x-y|^{3+2s}} dx dy \\
&\leq \iint_{\Omega_{n}\times \R^{3}} \frac{|v_{n}(x)-v_{n}(y)e^{\imath A_{\e}(\frac{x+y}{2})\cdot (x-y)}|^{2}}{|x-y|^{3+2s}} dx dy \\
&\quad+\iint_{\Omega_{n}\times \R^{3}} \frac{|w_{n}(x)-w_{n}(y)e^{\imath A_{\e}(\frac{x+y}{2})\cdot (x-y)}|^{2}}{|x-y|^{3+2s}} dx dy \rightarrow 0.
\end{split}\end{align}
In conclusion, putting together \eqref{UVW}, \eqref{Hn}, \eqref{Hn1234}, \eqref{Hn1}, \eqref{Hn2} and \eqref{Hn3}, we can deduce that the first identity in \eqref{3.9LG} holds true.
Concerning the second identity in \eqref{3.9LG}, from the definitions of $v_{n}$ and $w_{n}$ and using \eqref{TL} it follows that
\begin{align*}
\int_{\R^{3}}V_{\e}[|u_{n}|^{2}-|v_{n}|^{2}-|w_{n}|^{2}]dx&=\int_{\R^{3}} V_{\e}(1-\xi_{n}^{2}-(1-\xi_{n})^{2})|u_{n}|^{2}dx\\
&=\int_{\Omega_{n}} V_{\e}(1-\xi_{n}^{2}-(1-\xi_{n})^{2})|u_{n}|^{2}dx\\
&\leq \int_{\Omega_{n}} V_{\e}|u_{n}|^{2}dx\rightarrow 0.
\end{align*}
A similar argument shows that the fifth identity in \eqref{3.9LG} holds true.
Finally, we only prove the third identity since the fourth one can be obtained in a similar way.
By $(f_1)$-$(f_2)$, the H\"older inequality and \eqref{TL} we get
\begin{align*}
&\left| \int_{\R^{3}} F(|u_{n}|^{2}) - F(|v_{n}|^{2}) - F(|w_{n}|^{2})\, dx \right| \\
&= \left| \int_{\Omega_{n}} F(|u_{n}|^{2}) - F(|\xi_{n} u_{n}|^{2}) - F(|(1-\xi_{n})u_{n}|^{2})\, dx \right|\\
&\leq C \left( \e \int_{\Omega_{n}} |u_{n}|^{4} \, dx + \int_{\Omega_{n}} |u_{n}|^{q}\, dx \right) \\
&\leq C \left[ \e \left( \int_{\Omega_{n}} |u_{n}|^{2}\, dx \right)^{2r_{1}} \left(\int_{\Omega_{n}} |u_{n}|^{2^{*}_{s}}\, dx \right)^{\frac{4(1-r_{1})}{2^{*}_{s}}} + \left( \int_{\Omega_{n}} |u_{n}|^{2}\, dx \right)^{\frac{qr_{2}}{2}}\left( \int_{\Omega_{n}} |u_{n}|^{2^{*}_{s}}\, dx \right)^{\frac{q(1-r_{2})}{2^{*}_{s}}}\right] \\
&\rightarrow 0 \mbox{ as } n\rightarrow \infty, 
\end{align*}
where $r_{1}, r_{2}\in (0, 1)$ are such that $\frac{1}{4}= \frac{r_{1}}{2}+  \frac{1-r_{1}}{2^{*}_{s}}$ and $\frac{1}{q}= \frac{r_{2}}{2}+  \frac{1-r_{2}}{2^{*}_{s}}$. 
Therefore, we have proved that all identities in \eqref{3.9LG}  are true.

Now, we note that
\begin{align}\begin{split}\label{3.10LG}
[u_{n}]^{4}_{A_{\e}}&=([v_{n}]^{2}_{A_{\e}}+[w_{n}]^{2}_{A_{\e}}+o_{n}(1))^{2} \\
&=[v_{n}]^{4}_{A_{\e}}+[w_{n}]^{4}_{A_{\e}}+2[v_{n}]^{2}_{A_{\e}}[w_{n}]^{2}_{A_{\e}}+o_{n}(1) \\
&\geq [v_{n}]^{4}_{A_{\e}}+[w_{n}]^{4}_{A_{\e}} +o_{n}(1).
\end{split}\end{align}
Hence, in view of \eqref{3.9LG} and the definition of $\Phi(u)$, we can see that
$$
\Phi(u_{n})=\Phi(v_{n})+\Phi(w_{n})+o_{n}(1).
$$
Let us prove that
\begin{align}\label{3.8LG}
\liminf_{n\rightarrow \infty} \Phi(v_{n})\geq \alpha \quad \mbox{ and } \quad \liminf_{n\rightarrow \infty} \Phi(w_{n})\geq \ell-\alpha.
\end{align}
We only show the first relation of limit in \eqref{3.8LG} because the second one can be obtained in a similar way. We begin observing that
\begin{align}\begin{split}\label{ABCn}
&\iint_{\R^{6}} \frac{|v_{n}(x)-v_{n}(y)e^{\imath A_{\e}(\frac{x+y}{2})\cdot (x-y)}|^{2}}{|x-y|^{3+2s}} dx dy \\
&=\left(\iint_{B_{R_{n}}(y_{n})\times\R^{3}}+\iint_{\Omega_{n}\times \R^{3}}+\iint_{B_{2R_{n}}(y_{n})\times \R^{3}}\right) 
\frac{|v_{n}(x)-v_{n}(y)e^{\imath A_{\e}(\frac{x+y}{2})\cdot (x-y)}|^{2}}{|x-y|^{3+2s}} dx dy\\
&=:A_{n}+B_{n}+C_{n}.
\end{split}\end{align}
Using \eqref{39T} we can see that $B_{n}=o_{n}(1)$. On the other hand, being $v_{n}=0$ in $B^{c}_{2R_{n}}(y_{n})$, we can see that
\begin{align*}
C_{n}=\iint_{B^{c}_{2R_{n}}(y_{n})\times \Omega_{n}} \frac{|v_{n}(x)-v_{n}(y)e^{\imath A_{\e}(\frac{x+y}{2})\cdot (x-y)}|^{2}}{|x-y|^{3+2s}} dx dy+\iint_{B^{c}_{2R_{n}}(y_{n})\times B_{R_{n}}(y_{n})} \frac{|u_{n}(y)|^{2}}{|x-y|^{3+2s}} dx dy=o_{n}(1).
\end{align*}
Indeed, the first integral can be estimates arguing as in \eqref{Hn2}, while the second one as in \eqref{Hn1}.
Finally, we prove that
\begin{align}\label{AnN}
A_{n}=\iint_{B_{R_{n}}(y_{n})\times \R^{3}} \frac{|u_{n}(x)-u_{n}(y)e^{\imath A_{\e}(\frac{x+y}{2})\cdot (x-y)}|^{2}}{|x-y|^{3+2s}} dx dy+o_{n}(1).
\end{align}
Let us note that
\begin{align*}
A_{n}&=\iint_{B_{R_{n}}(y_{n})\times B_{R_{n}}(y_{n})} \frac{|u_{n}(x)-u_{n}(y)e^{\imath A_{\e}(\frac{x+y}{2})\cdot (x-y)}|^{2}}{|x-y|^{3+2s}} dx dy\\
&\quad +\iint_{B_{R_{n}}(y_{n})\times \Omega_{n}} \frac{|v_{n}(x)-v_{n}(y)e^{\imath A_{\e}(\frac{x+y}{2})\cdot (x-y)}|^{2}}{|x-y|^{3+2s}} dx dy \\
&\quad+\iint_{B_{R_{n}}(y_{n})\times B^{c}_{2R_{n}}(y_{n})} \frac{|u_{n}(x)|^{2}}{|x-y|^{3+2s}} dx dy \\
&=\iint_{B_{R_{n}}(y_{n})\times B_{R_{n}}(y_{n})} \frac{|u_{n}(x)-u_{n}(y)e^{\imath A_{\e}(\frac{x+y}{2})\cdot (x-y)}|^{2}}{|x-y|^{3+2s}} dx dy+o_{n}(1)
\end{align*}
in view of the estimates done for \eqref{Hn2} and \eqref{Hn1} respectively.
On the other hand
\begin{align*}
&\iint_{B_{R_{n}}(y_{n})\times \R^{3}} \frac{|u_{n}(x)-u_{n}(y)e^{\imath A_{\e}(\frac{x+y}{2})\cdot (x-y)}|^{2}}{|x-y|^{3+2s}} dx dy\\
&=\iint_{B_{R_{n}}(y_{n})\times B_{R_{n}}(y_{n})} \frac{|u_{n}(x)-u_{n}(y)e^{\imath A_{\e}(\frac{x+y}{2})\cdot (x-y)}|^{2}}{|x-y|^{3+2s}} dx dy\\
&\quad +\iint_{B_{R_{n}}(y_{n})\times \Omega_{n}} \frac{|u_{n}(x)-u_{n}(y)e^{\imath A_{\e}(\frac{x+y}{2})\cdot (x-y)}|^{2}}{|x-y|^{3+2s}} dx dy\\
&\quad +\iint_{B_{R_{n}}(y_{n})\times B^{c}_{2R_{n}}(y_{n})} \frac{|u_{n}(x)-u_{n}(y)e^{\imath A_{\e}(\frac{x+y}{2})\cdot (x-y)}|^{2}}{|x-y|^{3+2s}} dx dy.
%&=\iint_{B_{R_{n}}(y_{n})\times B_{R_{n}}(y_{n})} \frac{|u_{n}(x)-u_{n}(y)e^{\imath A_{\e}(\frac{x+y}{2})\cdot (x-y)}|^{2}}{|x-y|^{3+2s}} dx dy+o_{n}(1)
\end{align*}
We point out that the second term on the right hand side of the above identity, can be estimated as in \eqref{Hn2}, while for the third term, we  first  use the fact that $|u_{n}(x)-u_{n}(y)e^{\imath A_{\e}(\frac{x+y}{2})\cdot (x-y)}|^{2}\leq 2(|u_{n}(x)|^{2}+|u_{n}(y)|^{2})$ and then one argue as in \eqref{Hn1}. 
Therefore, 
\begin{align*}
&\iint_{B_{R_{n}}(y_{n})\times \R^{3}} \frac{|u_{n}(x)-u_{n}(y)e^{\imath A_{\e}(\frac{x+y}{2})\cdot (x-y)}|^{2}}{|x-y|^{3+2s}} dx dy\\
&=\iint_{B_{R_{n}}(y_{n})\times B_{R_{n}}(y_{n})} \frac{|u_{n}(x)-u_{n}(y)e^{\imath A_{\e}(\frac{x+y}{2})\cdot (x-y)}|^{2}}{|x-y|^{3+2s}} dx dy+o_{n}(1)
\end{align*}
and this implies that \eqref{AnN} is verified. Consequently,
\begin{align*}
\iint_{\R^{6}} \frac{|v_{n}(x)-v_{n}(y)e^{\imath A_{\e}(\frac{x+y}{2})\cdot (x-y)}|^{2}}{|x-y|^{3+2s}} dx dy =\iint_{B_{R_{n}}(y_{n})\times \R^{3}} \frac{|u_{n}(x)-u_{n}(y)e^{\imath A_{\e}(\frac{x+y}{2})\cdot (x-y)}|^{2}}{|x-y|^{3+2s}} dx dy+o_{n}(1). 
\end{align*}
Then, using $v_{n}=u_{n}$ in $B_{R_{n}}(y_{n})$, \eqref{Dun}  and the definition of $\rho_{n}$ we can deduce that  the first relation of limit in \eqref{3.8LG} holds true.\\
Hence, in the light of \eqref{3.8LG} we can infer that
$$
\ell=\lim_{n\rightarrow \infty} \Phi(u_{n})\geq \lim_{n\rightarrow \infty} \Phi(v_{n})+\lim_{n\rightarrow \infty} \Phi(w_{n})\geq \ell,
$$
that is
\begin{align}\label{3.11LG}
\lim_{n\rightarrow \infty} \Phi(v_{n})=\alpha \mbox{ and } \lim_{n\rightarrow \infty} \Phi(w_{n})=\ell-\alpha.
\end{align}
Taking into account \eqref{3.3LG}, \eqref{3.9LG} and \eqref{3.10LG} we deduce that
\begin{equation}\label{3.12LG}
0=\langle J'_{\e}(u_{n}), u_{n}\rangle+o_{n}(1)\geq \langle J'_{\e}(v_{n}), v_{n}\rangle+\langle J'_{\e}(w_{n}), w_{n}\rangle+o_{n}(1).
\end{equation}
Now, we distinguish two cases.\\
{\bf Case 1}. Up to subsequence, we may assume that either $\langle J'_{\e}(v_{n}), v_{n}\rangle\leq 0$ or $\langle J'_{\e}(w_{n}), w_{n}\rangle\leq 0$. Without loss of generality, suppose that  $\langle J'_{\e}(v_{n}), v_{n}\rangle\leq 0$.
Then
\begin{align}\label{3.13LG}
a[v_{n}]_{A_{\e}}^{2}+\int_{\R^{3}} V_{\e} |v_{n}|^{2}+b[v_{n}]^{4}_{A_{\e}}-\int_{\R^{3}} f(|v_{n}|^{2})|v_{n}|^{2}-|v_{n}|_{\2}^{\2}\leq 0.
\end{align}
Hence, for all $n\in \mathbb{N}$, there exists $t_{n}>0$ such that $t_{n}v_{n}\in \N_{\e}$ and $\langle J'_{\e}(t_{n}v_{n}), t_{n}v_{n}\rangle= 0$, that is
\begin{align}\label{3.14LG}
at_{n}^{2}[v_{n}]_{A_{\e}}^{2}+t_{n}^{2}\int_{\R^{3}} V_{\e} |v_{n}|^{2}+bt_{n}^{4}[v_{n}]^{4}_{A_{\e}}-\int_{\R^{3}} f(|t_{n}v_{n}|^{2})|t_{n}v_{n}|^{2}-t_{n}^{\2}|v_{n}|_{\2}^{\2}= 0.
\end{align} 
Combining \eqref{3.13LG} and \eqref{3.14LG} we can deduce that
\begin{align}
\left(1-\frac{1}{t_{n}^{2}} \right)\|v_{n}\|^{2}_{\e}+\int_{\R^{3}} \left(\frac{f(|t_{n}v_{n}|^{2})}{|t_{n}v_{n}|^{2}}-\frac{f(v_{n}|^{2})}{|v_{n}|^{2}}\right)|v_{n}|^{4}+(t_{n}^{\2-4}-1)|v_{n}|_{\2}^{\2}\leq 0,
\end{align}
which together with $(f_4)$ implies that $t_{n}\leq 1$. 
%Then, using $(f_3)$-$(f_4)$ and \eqref{3.11LG} we get
Then, using $t\mapsto \frac{1}{4}f(t)t-\frac{1}{2}F(t)$ is increasing for $t>0$ and \eqref{3.11LG} we get
\begin{align*}
c\leq J_{\e}(t_{n}v_{n})&=J_{\e}(t_{n}v_{n})-\frac{1}{4}\langle J'_{\e}(t_{n}v_{n}), t_{n}v_{n}\rangle \\
&=\frac{1}{4}t_{n}^{2}\|v_{n}\|_{\e}^{2}+\int_{\R^{3}} \frac{1}{4} f(|t_{n}v_{n}|^{2})|t_{n}v_{n}|^2-\frac{1}{2}F(|t_{n}v_{n}|^{2}) dx +\frac{4s-3}{12}t_{n}^{\2}|v_{n}|_{\2}^{\2} \\
&=\Phi(v_{n})\rightarrow \alpha<\ell=c
\end{align*}
and this gives a contradiction.

{\bf Case 2}. Up to a subsequence, we may assume that $\langle J'_{\e}(v_{n}), v_{n}\rangle >0$ and $\langle J'_{\e}(w_{n}), w_{n}\rangle >0$. \\
In view of \eqref{3.12LG}, we can deduce that $\langle J'_{\e}(v_{n}), v_{n}\rangle \rightarrow 0$ and $\langle J'_{\e}(w_{n}), w_{n}\rangle \rightarrow 0$ as $n\rightarrow \infty$. Using \eqref{3.9LG} and \eqref{3.10LG} we get
\begin{align}\label{3.15LG}
J_{\e}(u_{n})\geq J_{\e}(v_{n})+ J_{\e}(w_{n})+ o_{n}(1). 
\end{align}
If the sequence $(y_{n})\subset \R^{3}$ is bounded, we will deduce a contradiction by comparing $J_{\e}(w_{n})$ and $m_{V_{\infty}}$.
In this case, from assumption $(V)$, for any $\eta>0$ there exists $R_{0}>0$ such that 
\begin{align}\label{3.27TA}
V_{\e}(x)- V_{\infty}>-\eta, \quad \mbox{ for all  } |x|\geq \frac{R_{0}}{\e}. 
\end{align}
Since $(y_{n})\subset \R^{3}$ is bounded, there exists $R'>0$ such that $|y_{n}|\leq R'$. Thus 
$$
B^{c}_{R_{n}}(y_{n})\subset B^{c}_{R_{n}-R'}(0)\subset  B^{c}_{R_{0}/\e}(0)
$$ 
for $n$ large enough. Hence,  \eqref{3.27TA} yields
\begin{align*}
\int_{\R^{3}} (V_{\e}(x)- V_{\infty}) |w_{n}|^{2} \, dx &= \int_{|x-y_{n}|>R_{n}} (V_{\e}(x)- V_{\infty}) |w_{n}|^{2} \, dx \\
&\geq -\eta \int_{|x-y_{n}|>R_{n}} |w_{n}|^{2}\, dx \geq -C\eta, 
\end{align*}
and from the arbitrariness of $\eta$ we can infer that
\begin{align}\label{3.28TA}
\int_{\R^{3}} (V_{\e}(x) - V_{\infty})|w_{n}|^{2}\, dx \geq o_{n}(1).
\end{align}
Then, taking into account Lemma \ref{DI} and \eqref{3.28TA}  we get
\begin{align}\label{3.16LG1}
J_{\e}(w_{n})\geq E_{V_{\infty}}(|w_{n}|)+ o_{n}(1)
\end{align}
and 
\begin{align}\label{3.16LG2}
o_{n}(1)= \langle J'_{\e}(w_{n}), w_{n}\rangle \geq \langle E'_{V_\infty}(|w_{n}|), |w_{n}| \rangle +o_{n}(1). 
\end{align}
Arguing as in Case 1 and Lemma \ref{lemTA}, there exist two positive sequences $(t_{n}), (t'_{n})$ such that $t_{n}\leq 1$ for $n$ sufficiently large and $t'_{n}\rightarrow 1$ as $n\rightarrow \infty$, respectively, such that $t_{n}|w_{n}|\in \M_{V_{\infty}}$ and $t'_{n}v_{n}\in \N_{\e}$. \\
Using Lemma \ref{DI}, $t_{n}\leq 1$, $t\mapsto \frac{1}{4}f(t)t-\frac{1}{2}F(t)$ is increasing for $t>0$, we have 
\begin{align*}
J_{\e}(w_{n})&= J_{\e}(w_{n})- \frac{1}{4} \langle J'_{\e}(w_{n}), w_{n}\rangle +o_{n}(1)\\
&= \frac{1}{4}\|w_{n}\|_{\e}^{2} + \int_{\R^{3}} \frac{1}{4} f(|w_{n}|^{2}) |w_{n}|^{2} - \frac{1}{2}F(|w_{n}|^{2})\, dx + \frac{4s-3}{12} |w_{n}|_{2^{*}_{s}}^{2^{*}_{s}}\\
&\geq E_{V_\infty}(|w_{n}|)- \frac{1}{4}\langle E'_{V_\infty}(|w_{n}|), |w_{n}|\rangle + o_{n}(1)\\
&\geq E_{V_\infty}(t_{n}|w_{n}|) - \frac{1}{4}\langle E'_{V_\infty}(t_{n}|w_{n}|), t_{n}|w_{n}|\rangle + o_{n}(1)\\
&= E_{V_\infty}(t_{n}|w_{n}|)\geq c_{V_{\infty}} + o_{n}(1)
\end{align*}
and 
\begin{align*}
J_{\e}(v_{n})= J_{\e}(t'_{n}v_{n})+ o_{n}(1) \geq c_{\e} +o_{n}(1). 
\end{align*}
This combined with \eqref{3.15LG} yields $c\geq m_{V_{\infty}}+ c_{\e} \geq m_{V_{\infty}}$, which gives a contradiction. \\
If we suppose that $(y_{n})\subset \R^{3}$ is unbounded, we can argue as above to get a contradiction by comparing $J_{\e}(v_{n})$ and $m_{V_{\infty}}$. Consequently, the dichotomy does not happen.

{\bf Conclusion} From the above considerations and Proposition \ref{propL}, we can deduce that the sequence $(\rho_{n})$ is compact, that is there exists $(y_{n})\subset \R^{3}$ such that for every $\eta>0$ there exists $R>0$ such that $\int_{B_{R}^{c}(y_{n})} \rho_{n}\, dx <\eta$, which implies that 
\begin{align*}
\int_{B_{R}^{c}(y_{n})} V_{\e}(x) |u_{n}|^{2} + \frac{1}{4}[f(|u_{n}|^{2})|u_{n}|^{2}-2F(|u_{n}|^{2})] + |u_{n}|^{\2} \, dx <\eta. 
\end{align*}
In particular, by interpolation, we can see that for all fixed $r\in [2, \2]$ we get
\begin{align}\label{3.32TA}
\int_{ B^{c}_{R}(y_{n})} |u_{n}|^{r}dx&\leq \left(\int_{B^{c}_{R}(y_{n})} |u_{n}|^{2}dx\right)^{\frac{rr_{3}}{2}} \left(\int_{B^{c}_{R}(y_{n})} |u_{n}|^{\2}dx\right)^{\frac{r(1-r_{3})}{\2}}< C\eta,
\end{align}
for some $r_{3}\in (0, 1)$ such that $\frac{1}{r}=\frac{r_3}{2}+\frac{1-r_3}{\2}$. Hence, the sequence $(|u_{n}|^{r})$, with $r\in [2, \2]$, is compact. 

Now, we claim that $(y_{n})$ is bounded. Indeed, if this is not true, 
we can choose $r_{n}$ such that $|y_{n}|\geq r_{n}\geq R+\frac{R_{0}}{\e}$  with $r_n\rightarrow \infty$. For $n$ large enough, $B_{R}(y_{n})\subset B^{c}_{r_{n}-R}(0)\subset B^{c}_{R_{0}/\e}(0)$. In view of \eqref{3.32TA}
we can see that
\begin{align*}
\int_{\R^{3}} (V_{\e}(x)- V_{\infty}) |u_{n}|^{2}\, dx&=\int_{B_{R}(y_{n})} (V_{\e}(x)- V_{\infty}) |u_{n}|^{2}\, dx+\int_{B^{c}_{R}(y_{n})} (V_{\e}(x)- V_{\infty}) |u_{n}|^{2}\, dx \\
&\geq  -C\eta+o_{n}(1),
\end{align*}
and from the arbitrariness of $\eta$ we obtain 
$$
\int_{\R^{3}} (V_{\e}(x)- V_{\infty}) |u_{n}|^{2}\, dx\geq o_{n}(1).
$$
This fact together with Lemma \ref{DI} implies that $J_{\e}(u_{n})\geq E_{V_\infty}(|u_{n}|)+ o_{n}(1)$ and $\langle E'_{V_\infty}(|u_{n}|), |u_{n}|\rangle \leq \langle J'_{\e}(u_{n}), u_{n}\rangle =o_{n}(1)$. Arguing as in Case 1, we can find $t_{n}\leq 1$ for $n$ sufficiently large such that $t_{n}|u_{n}|\in \M_{V_{\infty}}$. Thus 
\begin{align*}
c&= J_{\e}(u_{n}) -\frac{1}{4} \langle J'_{\e}(u_{n}), u_{n}\rangle +o_{n}(1) \geq E_{V_\infty}(|u_{n}|) - \frac{1}{4}\langle E'_{V_\infty}(|u_{n}|), |u_{n}|\rangle + o_{n}(1)\\
&\geq E_{V_\infty}(t_{n}|u_{n}|) - \frac{1}{4}\langle E'_{V_\infty}(t_{n} |u_{n}|), t_{n}|u_{n}|\rangle + o_{n}(1)\\
&=E_{V_\infty}(t_{n}|u_{n}|)+ o_{n}(1)\geq m_{V_{\infty}}, 
\end{align*}
which gives a contradiction. \\
Therefore, $(y_{n})$ is bounded, and using \eqref{3.32TA} and Theorem \ref{Sembedding} we deduce that $|u_{n}|\rightarrow |u|$ in $L^{r}(\R^{3}, \R)$ for all $r\in [2, \2)$. This together with $(f_1)$-$(f_2)$ implies that the following splittings hold:
$$
\int_{\R^{3}} F(|h_{n}|^{2}) dx=\int_{\R^{3}} F(|u_{n}|^{2}) dx-\int_{\R^{3}} F(|u|^{2}) dx+o_{n}(1)=o_{n}(1)
$$
and
$$
\int_{\R^{3}} f(|h_{n}|^{2})|h_{n}|^{2} dx=\int_{\R^{3}} f(|u_{n}|^{2})|u_{n}|^{2} dx-\int_{\R^{3}} f(|u|^{2})|u|^{2} dx+o_{n}(1)=o_{n}(1)
$$
where $h_{n}:=u_{n}-u$. Then, applying Brezis-Lieb Lemma \cite{BL}, \eqref{B^2} and \eqref{ieun}, we have
\begin{align}\begin{split}\label{3.17LG}
c+\frac{bB^{4}}{4}-I_{\e}(u)+o_{n}(1)&=I_{\e}(u_{n})-I_{\e}(u)\\
&=\frac{(a+bB^{2})}{2}[h_{n}]^{2}_{A_{\e}}+\frac{1}{2}\int_{\R^{3}}V_{\e}|h_{n}|^{2}dx-\frac{1}{\2}|h_{n}|_{\2}^{\2}+o_{n}(1)
\end{split}\end{align}
and
\begin{align}\label{3.18LG}
o_{n}(1)=\langle I_{\e}'(u_{n}), u_{n}\rangle-\langle I'_{\e}(u), u\rangle=(a+bB^{2})[h_{n}]^{2}_{A_{\e}}+\int_{\R^{3}}V_{\e}|h_{n}|^{2}dx-|h_{n}|_{\2}^{\2}+o_{n}(1).
\end{align}
Here $I_{\e}$ is the functional defined in \eqref{defIe}.
We notice that in \eqref{3.18LG} we used the fact that $u$ is a critical point of $I_{\e}$ by the weak convergence and \eqref{B^2}. 
In particular, $(f_3)$ and $s>\frac{3}{4}$ give
\begin{align*}
I_{\e}(u)=I_{\e}(u)-\frac{1}{4}\langle I'_{\e}(u), u\rangle&=\frac{1}{4}(a+bB^{2})[u]^{2}_{\e}+\int_{\R^{3}}V_{\e}|u|^{2}dx\\
&\quad+\int_{\R^{3}}\frac{1}{4}f(|u|^{2})|u|^{2}-\frac{1}{2}F(|u|^{2})dx+\frac{4s-3}{12}|u|_{\2}^{\2}\geq 0.
\end{align*} 
Now, we may assume that $|h_{n}|_{\2}^{\2}\rightarrow \mathcal{L}^{3}\geq 0$. Suppose that $\mathcal{L}>0$. Then, arguing as in the Step 1, we can see that \eqref{3.18LG} and Sobolev embedding yield
$$
aS_{*}\mathcal{L}^{3-2s}+bS_{*}^{2}\mathcal{L}^{6-4s}-\mathcal{L}^{3}\leq 0,
$$
that is $\mathcal{L}\geq T$, where $T$ is the unique maximum point of $K(t)$ which is defined as in \eqref{defK}. Thus, taking into account \eqref{3.17LG},  \eqref{3.18LG} and $I_{\e}(u)\geq 0$, we can repeat the same calculations done in \eqref{LT} to infer that 
\begin{align*}
c&\geq I_{\e}(u_{n})-\frac{1}{\2}\langle I_{\e}'(u_{n}), u_{n}\rangle-\frac{bB^{4}}{4}+o_{n}(1)\\
&\geq \left(\frac{a}{2}-\frac{a}{\2}\right)S_{*}|h_{n}|_{\2}^{2}+\left(\frac{b}{4}-\frac{b}{\2}\right)S^{2}_{*}|h_{n}|_{\2}^{4}+o_{n}(1) \\
&=\left(\frac{a}{2}-\frac{a}{\2}\right)S_{*}\mathcal{L}^{3-2s}+\left(\frac{b}{4}-\frac{b}{\2}\right)S^{2}_{*}\mathcal{L}^{6-4s}\\
&\geq \left(\frac{a}{2}-\frac{a}{\2}\right)S_{*}T^{3-2s}+\left(\frac{b}{4}-\frac{b}{\2}\right)S^{2}_{*}T^{6-4s} \\
&=\frac{a}{2}S_{*}T^{3-2s}+\frac{b}{4}T^{6-4s}-\frac{1}{\2}T^{3}=c_{*}
\end{align*}
that is a contradiction. Hence $\mathcal{L}=0$, and using \eqref{3.18LG} we can conclude that  $\|h_{n}\|_{\e}\rightarrow 0$, that is $u_{n}\rightarrow u$ in $\h$.

Now, we consider the case $V_{\infty}=\infty$. 
In view of Theorem \ref{Sembedding} we know that $(|u_{n}|)$ strongly converges in $L^{r}(\R^{3}, \R)$ for all $r\in [2, \2)$. Then, setting $h_{n}=u_{n}-u$, we can argue as above to deduce that $u_{n}\rightarrow u$ in $\h$.
\end{proof}

\noindent
We are now ready to prove our main compactness result of this section.
\begin{prop}\label{propPSc}
Let $c\in \R_{+}$ be such that $c<m_{V_{\infty}}$ if $V_{\infty}<\infty$, and $c<c_{*}$ if $V_{\infty}=\infty$. Then, the functional $J_{\e}$ restricted to $\mathcal{N}_{\e}$ satisfies the $(PS)_{c}$ condition at the level $c$.
\end{prop}
\begin{proof}
Let $(u_{n})\subset \mathcal{N}_{\e}$ be such that $J_{\e}(u_{n})\rightarrow c$ and $\|J'_{\e}(u_{n})_{|\mathcal{N}_{\e}}\|_{*}=o_{n}(1)$. Then (see \cite{W}) we can find $(\lambda_{n})\subset \R$ such that
\begin{equation}\label{AFT}
J'_{\e}(u_{n})=\lambda_{n} T'_{\e}(u_{n})+o_{n}(1),
\end{equation}
where $T_{\e}: \h\rightarrow \R$ is defined as
\begin{align*}
T_{\e}(u)=\|u\|_{\e}^{2}+b[u]^{4}_{A_{\e}}-\int_{\R^{3}} f(|u|^{2})|u|^{2}+|u|^{\2}\, dx.
\end{align*}
In view of $\langle J'_{\e}(u_{n}), u_{n}\rangle=0$, $(f_4)$ and $s>\frac{3}{4}$ we obtain
\begin{align}
&\langle T'_{\e}(u_{n}), u_{n}\rangle \nonumber\\
&=2\|u_{n}\|_{\e}^{2}+4 b[u_{n}]^{4}_{A_{\e}}-2\int_{\R^{3}} f'(|u_{n}|^{2})|u_{n}|^{4}\, dx-2\int_{\R^{3}} f(|u_{n}|^{2})|u_{n}|^{2}\, dx-\2|u_{n}|_{\2}^{\2} \nonumber\\
&=-2\|u_{n}\|^{2}_{\e}+2\int_{\R^{3}} f(|u_{n}|^{2})|u_{n}|^{2}\, dx-2\int_{\R^{3}} f'(|u_{n}|^{2})|u_{n}|^{4}\, dx-(\2-4)|u_{n}|_{\2}^{\2} \nonumber \\
&\leq -(\2-4)\int_{\R^{3}} |u_{n}|^{\2} dx< 0 \label{22ZS1}.
\end{align}
Since $(u_{n})$ is bounded in $\h$, we may assume that $\langle T'_{\e}(u_{n}), u_{n}\rangle\rightarrow \ell\leq 0$. If $\ell=0$, from \eqref{22ZS1} it follows that $|u_{n}|\rightarrow 0$ in $L^{\2}(\R^{3}, \R)$. By interpolation, we can see that $|u_{n}|\rightarrow 0$ in $L^{p}(\R^{3}, \R)$ for all $p\in [2, \2]$. This fact combined with $\langle J'_{\e}(u_{n}), u_{n}\rangle=0$ and $(f_1)$-$(f_2)$ implies that
\begin{equation*}
\|u_{n}\|^{2}_{\e}\leq \int_{\R^{3}} f(|u_{n}|^{2}) |u_{n}|^{2} \,dx+|u_{n}|_{\2}^{\2}= o_{n}(1)
\end{equation*}
that is $\|u_{n}\|_{\e}\rightarrow 0$ which contradicts \eqref{uNr}. Consequently, $\ell<0$ and in the light of \eqref{AFT} we can deduce that $\lambda_{n}\rightarrow 0$. Hence, $(u_{n})$ is a $(PS)_{c}$ sequence for the unconstrained functional and we can apply Lemma \ref{PSc} to get the thesis.
\end{proof}

\noindent
As a byproduct of the above proof we have the following result.
\begin{cor}\label{cor}
The critical points of the functional $J_{\e}$ on $\mathcal{N}_{\e}$ are critical points of $J_{\e}$.
\end{cor}

\section{Existence result for \eqref{Pe}}
In view of Proposition \ref{propPSc} we can establish a first existence result for \eqref{Pe} provided that $\e>0$ is sufficiently small. More precisely, we obtain:
\begin{thm}\label{AMlem1}
Assume that $(V)$ and $(f_1)$-$(f_4)$ hold. Then there exists $\e^{*}>0$ such that for any $\e\in (0, \e^{*})$, problem \eqref{Pe} admits a ground state solution.
\end{thm}
\begin{proof}
By Lemma \ref{MPG} we know that $J_{\e}$ has a mountain pass geometry, so, using a version of the mountain pass theorem without $(PS)$ condition (see \cite{W}), there exists a $(PS)_{c_{\e}}$ sequence $(u_{n})\subset \h$ for $J_{\e}$.
By Lemma \ref{lem2.3HZ} we know that $(u_{n})$ is bounded in $\h$ so we may assume that $u_{n}\rightharpoonup u$ in $\h$. \\
Firstly, we consider the case $V_{\infty}<\infty$. In view of Proposition \ref{propPSc} it is enough to show that $c_{\e}<m_{V_{\infty}}$ for $\e>0$ small enough. Without loss of generality, we may suppose that 
$$
V(0)=V_{0}=\inf_{x\in \R^{3}} V(x).
$$
Let $\mu\in  (V_{0}, V_{\infty})$. Clearly $m_{V_{0}}<m_{\mu}<m_{V_{\infty}}$. 
In the light of Lemma \ref{FS}, we can find a positive ground state $w\in H^{s}_{V_{0}}$ to \eqref{AP0}, that is $E'_{V_{0}}(w)=0$ and $E_{V_{0}}(w)=m_{V_{0}}$. 
Since $w\in C^{1, \gamma}(\R^{3}, \R)\cap L^{\infty}(\R^{3},\R)$, for some $\gamma>0$, we get $|w(x)|\rightarrow 0$ as $|x|\rightarrow \infty$. Observing that $w$ satisfies
$$
(-\Delta)^{s}w+\frac{V_{0}}{a+bM^{2}}w=(a+b[w]^{2})^{-1}[f(w^{2})w+w^{\2-1}-V_{0} w]+\frac{V_{0}}{a+bM^{2}}w \mbox{ in } \R^{3},
$$
where $0<a\leq a+b[w]^{2}\leq a+bM^{2}$, we can argue as in Lemma $4.3$ in \cite{FQT} to deduce the following decay estimate
\begin{equation}\label{remdecay}
0<w(x)\leq \frac{C}{|x|^{3+2s}} \quad \mbox{ for } |x|>>1.
\end{equation}
Let $\eta\in C^{\infty}_{c}(\R^{3}, \R)$ be a cut-off function such that $\eta=1$ in $B_{1}(0)$ and $\eta=0$ in $B_{2}^{c}(0)$. Let us define $w_{r}(x):=\eta_{r}(x)w(x) e^{\imath A(0)\cdot x}$, with $\eta_r(x)=\eta(x/r)$ for $r>0$, and we observe that $|w_{r}|=\eta_{r}w$ and $w_{r}\in \h$ in view of Lemma \ref{aux}.
Take $t_{r}>0$ such that 
\begin{equation*}
E_{\mu}(t_{r} |w_{r}|)=\max_{t\geq 0} E_{\mu}(t |w_{r}|)
\end{equation*}
Let us prove that there exists $r$ sufficiently large such that $E_{\mu}(t_{r}|w_{r}|)<m_{V_{\infty}}$.\\
If by contradiction $E_{\mu}(t_{r}|w_{r}|)\geq m_{V_{\infty}}$ for any $r>0$, using the fact that $|w_{r}|\rightarrow w$ in $H^{s}(\R^{3}, \R)$ as $r\rightarrow \infty$ (see Lemma 5 in \cite{PP}), we have $t_{r}\rightarrow 1$ and
$$
m_{V_{\infty}}\leq \liminf_{r\rightarrow \infty} E_{\mu}(t_{r}|w_{r}|)=E_{\mu}(w)=m_{\mu}
$$
which gives a contradiction since $m_{V_{\infty}}>m_{\mu}$.
Hence, there exists $r>0$ such that
\begin{align}\label{tv19}
E_{\mu}(t_{r}|w_{r}|)=\max_{\tau\geq 0} E_{\mu}(\tau (t_{r} |w_{r}|)) \mbox{ and } E_{\mu}(t_{r}|w_{r}|)<m_{V_{\infty}}.
\end{align}
Now, we show that 
\begin{equation}\label{limwr}
\lim_{\e\rightarrow 0}[w_{r}]^{2}_{A_{\e}}=[\eta_{r}w]^{2}.
\end{equation}
Firstly, we note that 
\begin{align*}
[w_{r}]_{A_{\e}}^{2}
&=\iint_{\R^{6}} \frac{|e^{\imath A(0)\cdot x}\eta_{r}(x)w(x)-e^{\imath A_{\e}(\frac{x+y}{2})\cdot (x-y)}e^{\imath A(0)\cdot y} \eta_{r}(y)w(y)|^{2}}{|x-y|^{3+2s}} dx dy \nonumber \\
&=[\eta_{r} w]^{2}
+\iint_{\R^{6}} \frac{\eta_{r}^2(y)w^2(y) |e^{\imath [A_{\e}(\frac{x+y}{2})-A(0)]\cdot (x-y)}-1|^{2}}{|x-y|^{3+2s}} dx dy\\
&\quad+2\Re \iint_{\R^{6}} \frac{(\eta_{r}(x)w(x)-\eta_{r}(y)w(y))\eta_{r}(y)w(y)(1-e^{-\imath [A_{\e}(\frac{x+y}{2})-A(0)]\cdot (x-y)})}{|x-y|^{3+2s}} dx dy \\
&=: [\eta_{r} w]^{2}+X_{\e}+2Y_{\e}.
\end{align*}
Since 
$|Y_{\e}|\leq [\eta_{r} w] \sqrt{X_{\e}}$, it s enough to show that	$X_{\e}\rightarrow 0$ as $\e\rightarrow 0$ to deduce that \eqref{limwr} holds.\\
Let us note that, for $0<\beta<\alpha/({1+\alpha-s})$, it holds
\begin{equation}\label{Ye}
\begin{split}
X_{\e}&\leq \int_{\R^{3}} w^{2}(y) dy \int_{|x-y|\geq\e^{-\beta}} \frac{|e^{\imath [A_{\e}(\frac{x+y}{2})-A(0)]\cdot (x-y)}-1|^{2}}{|x-y|^{3+2s}} dx\\
&+\int_{\R^{3}} w^{2}(y) dy  \int_{|x-y|<\e^{-\beta}} \frac{|e^{\imath [A_{\e}(\frac{x+y}{2})-A(0)]\cdot (x-y)}-1|^{2}}{|x-y|^{3+2s}} dx \\
&=:X^{1}_{\e}+X^{2}_{\e}.
\end{split}
\end{equation}
Since $|e^{\imath t}-1|^{2}\leq 4$ and recalling that $w\in H^{s}_{\mu}$, we can see that 
\begin{equation}\label{Ye1}
X_{\e}^{1}\leq C \int_{\R^{3}} w^{2}(y) dy \int_{\e^{-\beta}}^\infty \rho^{-1-2s} d\rho\leq C \e^{2\beta s} \rightarrow 0.
\end{equation}
Concerning $X^{2}_{\e}$, since $|e^{\imath t}-1|^{2}\leq t^{2}$ for all $t\in \R$, $A\in C^{0,\alpha}(\R^3,\R^3)$ for $\alpha\in(0,1]$, and $|x+y|^{2}\leq 2(|x-y|^{2}+4|y|^{2})$,	 we have
\begin{equation}\label{Ye2}
\begin{split}
X^{2}_{\e}&
\leq \int_{\R^{3}} w^{2}(y) dy  \int_{|x-y|<\e^{-\beta}} \frac{|A_{\e}\left(\frac{x+y}{2}\right)-A(0)|^{2} }{|x-y|^{3+2s-2}} dx \\
&\leq C\e^{2\alpha} \int_{\R^{3}} w^{2}(y) dy  \int_{|x-y|<\e^{-\beta}} \frac{|x+y|^{2\alpha} }{|x-y|^{3+2s-2}} dx \\
&\leq C\e^{2\alpha} \left(\int_{\R^{3}} w^{2}(y) dy  \int_{|x-y|<\e^{-\beta}} \frac{1 }{|x-y|^{3+2s-2-2\alpha}} dx\right.\\
&\qquad\qquad+ \left. \int_{\R^{3}} |y|^{2\alpha} w^{2}(y) dy  \int_{|x-y|<\e^{-\beta}} \frac{1}{|x-y|^{3+2s-2}} dx\right) \\
&=: C\e^{2\alpha} (X^{2, 1}_{\e}+X^{2, 2}_{\e}).
\end{split}
\end{equation}	
Hence
\begin{equation}\label{Ye21}
X^{2, 1}_{\e}
= C  \int_{\R^{3}} w^{2}(y) dy \int_0^{\e^{-\beta}} \rho^{1+2\alpha-2s} d\rho
\leq C\e^{-2\beta(1+\alpha-s)}.
\end{equation}
On the other hand, using \eqref{remdecay}, we infer that
\begin{equation}\label{Ye22}
\begin{split}
 X^{2, 2}_{\e}
 &\leq C  \int_{\R^{3}} |y|^{2\alpha} w^{2}(y) dy \int_0^{\e^{-\beta}}\rho^{1-2s} d\rho  \\
&\leq C \e^{-2\beta(1-s)} \left[\int_{B_1(0)}  w^{2}(y) dy + \int_{B_1^c(0)} \frac{1}{|y|^{2(3+2s)-2\alpha}} dy \right]  \\
&\leq C \e^{-2\beta(1-s)}.
\end{split}
\end{equation}
Taking into account \eqref{Ye}, \eqref{Ye1}, \eqref{Ye2}, \eqref{Ye21} and \eqref{Ye22} we can conclude that $X_{\e}\rightarrow 0$.\\
Now, in view of condition $(V)$, there exists  $\e_{0}>0$ such that 
\begin{equation}\label{tv20}
V_{\e} (x)\leq \mu \mbox{ for all } x\in \supp(|w_{r}|), \e\in (0, \e_{0}).
\end{equation} 
Therefore, putting together \eqref{tv19} , \eqref{limwr} and \eqref{tv20}, 
we deduce that
$$
\limsup_{\e\rightarrow 0}c_{\e}\leq \limsup_{\e\rightarrow 0}\left[\max_{\tau\geq 0} J_{\e}(\tau (t_{r} w_{r}))\right]\leq \max_{\tau\geq 0} E_{\mu}(\tau (t_{r} |w_{r}|))=E_{\mu}(t_{r}|w_{r}|)<m_{V_{\infty}}
$$ 
which implies that $c_{\e}<m_{V_{\infty}}$ for any $\e>0$ sufficiently small.\\
Secondly, we assume that $V_{\infty}=\infty$. Then, using Lemma \ref{Sembedding}, we know that $(|u_{n}|)$ strongly converges in $L^{r}(\R^{3}, \R)$ for all $r\in [2, \2)$. Arguing as in the last part of Lemma \ref{PSc} we can deduce that $u_{n}\rightarrow u$ in $\h$, and consequently $J_{\e}(u)=c_{\e}$ and $J'_{\e}(u)=0$, where $u\in \h$ is the weak limit of $u_{n}$.
\end{proof}

\section{Proof of Theorem \ref{thm1}}
This last section is devoted to the proof of the main result of this work. Indeed, we apply the Ljusternik-Schnirelmann category theory to obtain multiple solutions to \eqref{Pe}. 
In particular, we relate the number of solutions of \eqref{Pe} to the topology of the set $M$.

Firstly, we fix some notation and prove some preliminary lemmas.

Let $\eta\in C^{\infty}_{0}(\R_{+}, [0, 1])$ be such that $\eta(t)=1$ if $0\leq t\leq \frac{1}{2}$ and $\eta(t)=0$ if $t\geq 1$.
For any $y\in M$, we introduce (see \cite{AD})
$$
\Psi_{\e, y}(x)=\eta(|\e x-y|) w\left(\frac{\e x-y}{\e}\right)e^{\imath \tau_{y} \left( \frac{\e x-y}{\e} \right)},
$$
where $\tau_{y}(x)=\sum_{j=1}^{3} A_{j}(y)x_{j}$ and $w\in H^{s}_{V_{0}}$ is a positive ground state solution to  autonomous problem \eqref{AP0} (see Lemma \ref{FS}), and
denote by $t_{\e}>0$ the unique number satisfying 
$$
\max_{t\geq 0} J_{\e}(t \Psi_{\e, y})=J_{\e}(t_{\e} \Psi_{\e, y}). 
$$
Let $\Phi_{\e}: M\rightarrow \N_{\e}$ be given by
$$
\Phi_{\e}(y)= t_{\e} \Psi_{\e, y}.
$$
\begin{lem}\label{lem3.4}
The functional $\Phi_{\e}$ satisfies the following limit
\begin{equation*}
\lim_{\e\rightarrow 0} J_{\e}(\Phi_{\e}(y))=m_{V_{0}} \mbox{ uniformly in } y\in M.
\end{equation*}
\end{lem}
\begin{proof}
Assume by contradiction that there exist $\delta_{0}>0$, $(y_{n})\subset M$ and $\e_{n}\rightarrow 0$ such that 
\begin{equation}\label{puac}
|J_{\e_{n}}(\Phi_{\e_{n}}(y_{n}))-m_{V_{0}}|\geq \delta_{0}.
\end{equation}
Applying Lemma $4.1$ in \cite{AD} and the Dominated Convergence Theorem we can observe that 
\begin{align}\begin{split}\label{nio3}
&\| \Psi_{\e_{n}, y_{n}} \|^{2}_{\e_{n}}\rightarrow \|w\|^{2}_{V_{0}}\in (0, \infty). 
\end{split}\end{align}
Now we prove that $t_{\e_{n}}\rightarrow 1$. Indeed, using $\langle J'_{\e_{n}}(\Phi_{\e_{n}}(y_{n})),\Phi_{\e_{n}}(y_{n})\rangle=0$ and \eqref{uNr} we get
\begin{align}\label{1nio}
r^{2} &\leq t_{\e_{n}}^{2}\|\Psi_{\e_{n}, y_{n}}\|_{\e_{n}}^{2} +b t_{\e_{n}}^{4} [\Psi_{\e_{n}, y_{n}}]^{4}_{A_{\e_{n}}} \nonumber \\
&=\int_{\R^{3}} f(|t_{\e_{n}} \Psi_{\e_{n}, y_{n}}|^{2}) \, |t_{\e_{n}} \Psi_{\e_{n}, y_{n}}|^{2} \, dx+ t_{\e_{n}}^{\2} |\Psi_{\e_{n}, y_{n}}|^{\2}_{\2}, 
\end{align}
which together with $(f_{1})$-$(f_{2})$ implies that $t_{\e_{n}}\nrightarrow 0$, so that $t_{\e_{n}}\geq t_{0}>0$ for some $t_{0}>0$. \\
If $t_{\e_{n}}\rightarrow \infty$, then we can see 
\begin{align}\label{nioo}
\frac{1}{t_{\e_{n}}^{2}} \|\Psi_{\e_{n}, y_{n}}\|_{\e_{n}}^{2}+ b[\Psi_{\e_{n}, y_{n}}]^{4}_{A_{\e_{n}}} 
&=\int_{\R^{3}} \frac{f(|t_{\e_{n}}\Psi_{\e_{n},y_{n}}|^{2})}{|t_{\e_{n}}\Psi_{\e_{n}, y_{n}}|^{2}}  |\Psi_{\e_{n}, y_{n}}|^{4}dx + t_{\e_{n}}^{\2-4} |\Psi_{\e_{n}, y_{n}}|^{\2}_{\2}\nonumber \\
&> \int_{B_{\frac{1}{2}}(0)} \frac{f(|t_{\e_{n}}\eta(|\e_{n}z|)w(z)|^{2})}{|t_{\e_{n}}\eta(|\e_{n}z|)w(z)|^{2}} (\eta(|\e_{n}z|)w(z))^{4}dz \nonumber \\
&=\int_{B_{\frac{1}{2}}(0)} \frac{f(t^{2}_{\e_{n}}w(z)^{2})}{t^{2}_{\e_{n}}w(z)^{2}} w(z)^{4}dz \nonumber \\
&\geq \frac{f(t^{2}_{\e_{n}}w(\hat{z})^{2})}{t_{\e_{n}}^{2}w(\hat{z})^{2}}w(\hat{z})^{4} |B_{\frac{1}{2}}(0)|, 
\end{align}
where
\begin{equation*}
w(\hat{z}):=\min_{z\in \overline{B}_{\frac{1}{2}}(0)} w(z)>0.
\end{equation*} 
This fact together with \eqref{nio3}, \eqref{nioo} and $(f_{4})$ yields 
$$
b[w]^{4}=\infty,
$$
that is a contradiction. Hence, $0<t_{0}\leq t_{\e_{n}} \leq C$, and we may assume that $t_{\e_{n}}\rightarrow T>0$. Now we show that $T=1$.  
Let us observe that by the Dominated Convergence Theorem we can see 
\begin{align*}
&\lim_{n\rightarrow \infty} \int_{\R^{3}} F(|\Psi_{\e_{n}, y_{n}}|^{2}) = \int_{\R^{3}} F(w^{2}), \\
&\lim_{n\rightarrow \infty} \int_{\R^{3}} f(|\Psi_{\e_{n}, y_{n}}|^{2})\, |\Psi_{\e_{n}, y_{n}}|^{2} = \int_{\R^{3}} f(w^{2}) w^{2}, \\
&\lim_{n\rightarrow \infty} |\Psi_{\e_{n}, y_{n}}|_{\2}^{\2} = |w|_{\2}^{\2}. 
\end{align*}
Therefore, taking the limit as $n\rightarrow \infty$ in \eqref{1nio}, we can deduce that
\begin{align*}
\frac{1}{T^{2}}\|w\|^{2}_{V_{0}}+b[w]^{4}=\int_{\R^{3}} \frac{f((T w)^{2})}{(Tw)^{2}} \,w^{4} + T^{\2-4} |w|_{\2}^{\2}.
\end{align*}
Since $w\in \M_{V_{0}}$ and $\frac{f(t)}{t}$ is increasing by $(f_4)$, we can infer that $T=1$.
Letting the limit as $n\rightarrow \infty$ and using $t_{\e_{n}}\rightarrow 1$ 
we can conclude that
$$
\lim_{n\rightarrow \infty} J_{\e_{n}}(\Phi_{\e_{n}, y_{n}})=E_{V_{0}}(w)=m_{V_{0}},
$$
which gives a contradiction in view of \eqref{puac}.
\end{proof}

\noindent
For any $\delta>0$, we take $\rho=\rho(\delta)>0$ such that $M_{\delta}\subset B_{\rho}(0)$. Define $\varUpsilon: \R^{3}\rightarrow \R^{3}$ as follows:
\begin{equation*}
\varUpsilon(x)=
\left\{
\begin{array}{ll}
x &\mbox{ if } |x|<\rho \\
\frac{\rho x}{|x|} &\mbox{ if } |x|\geq \rho.
\end{array}
\right.
\end{equation*}
Let us consider the barycenter map $\beta_{\e}: \N_{\e}\rightarrow \R^{3}$ given by
\begin{align*}
\beta_{\e}(u)=\frac{\displaystyle{\int_{\R^{3}} \varUpsilon(\e x)|u(x)|^{2} \,dx}}{\displaystyle{\int_{\R^{3}} |u(x)|^{2} \,dx}}.
\end{align*}
Then, we can prove the following result:
\noindent
\begin{lem}\label{lem3.5N}
The function $\beta_{\e}$ verifies the following limit:
\begin{equation*}
\lim_{\e \rightarrow 0} \beta_{\e}(\Phi_{\e}(y))=y \mbox{ uniformly in } y\in M.
\end{equation*}
\end{lem}

\begin{proof}
Assume by contradiction that there are $k>0$, $(y_{n})\subset M$ and $\e_{n}\rightarrow 0$ such that
\begin{align}\label{4.7AdA}
|\beta_{\e_{n}} (\Phi_{\e_{n}} (y_{n}))- y_{n}|\geq k. 
\end{align}
Using the change of variable $z= \frac{\e_{n}x- y_{n}}{\e_{n}}$, we have 
\begin{align*}
\beta_{\e_{n}}(\Phi_{\e_{n}}(y_{n})) = y_{n} + \frac{\int_{\R^{3}} [\Upsilon(\e_{n}z+y_{n})-y_{n}] |\eta(|\e_{n}z|)|^{2} |w(z)|^{2} dz}{\int_{\R^{3}} |\eta(|\e_{n} z|)|^{2} |w(z)|^{2}\, dz}. 
\end{align*}
Thus, by $(y_{n})\subset M\subset M_{\delta} \subset B_{\rho}(0)$ and applying the Dominated Convergence Theorem we get
\begin{align*}
|\beta_{\e_{n}} (\Phi_{\e_{n}}(y_{n})) - y_{n}|=o_{n}(1)
\end{align*}
which contradicts \eqref{4.7AdA}. 
\end{proof}

\noindent
Next, we establish the following technical result:
\begin{lem}\label{prop3.3}
Let $\e_{n}\rightarrow 0$ and $(u_{n})\subset \mathcal{N}_{\e_{n}}$ be such that $J_{\e_{n}}(u_{n})\rightarrow m_{V_{0}}$. Then there exists $(\tilde{y}_{n})\subset \R^{3}$ such that $v_{n}(x)=|u_{n}|(x+\tilde{y}_{n})$ has a convergent subsequence in $H^{s}_{V_{0}}$. Moreover, up to a subsequence, $y_{n}=\e_{n} \tilde{y}_{n}\rightarrow y_{0}$ for some $y_{0}\in M$.
\end{lem}
\begin{proof}
Taking into account $\langle J'_{\e_{n}}(u_{n}), u_{n}\rangle=0$, $J_{\e_{n}}(u_{n})= c_{V_{0}}+o_{n}(1)$, $c_{V_{0}}>0$,  and arguing as in Lemma \ref{lem2.3HZ}, we can see that $\|u_{n}\|_{\e_{n}}\leq C$ for all $n\in \mathbb{N}$ and $\|u_{n}\|_{\e_{n}} \nrightarrow 0$. Moreover, from Lemma \ref{DI}, we also know that $(|u_{n}|)$ is bounded in $H^{s}_{V_{0}}$. Then, proceeding as in the first part of Lemma \ref{PSc}, we can find a sequence $(\tilde{y}_{n})\subset \R^{3}$, and constants $R>0$ and $\beta>0$ such that
\begin{equation}\label{sacchi}
\liminf_{n\rightarrow \infty}\int_{B_{R}(\tilde{y}_{n})} |u_{n}|^{2} \, dx\geq \beta>0.
\end{equation}
Put $v_{n}(x)=|u_{n}|(x+\tilde{y}_{n})$. Hence, $(v_{n})$ is bounded in $H^{s}_{V_{0}}$ and we may assume that 
$v_{n}\rightharpoonup v\not\equiv 0$ in $H^{s}_{V_{0}}$  as $n\rightarrow \infty$.
Fix $t_{n}>0$ such that $\tilde{v}_{n}=t_{n} v_{n}\in \mathcal{M}_{V_{0}}$. Using Lemma \ref{DI}, we can deduce that 
$$
m_{V_{0}}\leq E_{V_{0}}(\tilde{v}_{n})\leq \max_{t\geq 0}J_{\e_{n}}(tu_{n})= J_{\e_{n}}(u_{n})= m_{V_{0}}+ o_{n}(1)
$$
which implies that $E_{V_{0}}(\tilde{v}_{n})\rightarrow m_{V_{0}}$. 
Since $v_{n}\nrightarrow 0$ in $H^{s}_{V_{0}}$ and $(\tilde{v}_{n})$ is bounded in $H^{s}_{V_{0}}$, we deduce that $(t_{n})$ is bounded in $\R$, and $t_{n}\rightarrow t^{*}\geq 0$. Indeed $t^{*}>0$, otherwise, if $t^{*}=0$, then $\tilde{v}_{n}\rightarrow 0$  in $H^{s}_{V_{0}}$ and $E_{V_{0}}(\tilde{v}_{n})\rightarrow 0$ which is impossible because $m_{V_{0}}>0$. From the uniqueness of the weak limit, we can deduce that $\tilde{v}_{n}\rightharpoonup \tilde{v}=t^{*}v\not\equiv 0$ in $H^{s}_{V_{0}}$. 
This combined with Lemma \ref{FSC} implies that
\begin{equation}\label{elena}
\tilde{v}_{n}\rightarrow \tilde{v} \mbox{ in } H^{s}_{V_{0}}.
\end{equation} 
Consequently, $v_{n}\rightarrow v$ in $H^{s}_{V_{0}}$ as $n\rightarrow \infty$.

Now, we set $y_{n}=\e_{n}\tilde{y}_{n}$ and we show that $(y_{n})$ admits a subsequence, still denoted by $y_{n}$, such that $y_{n}\rightarrow y_{0}$ for some $y_{0}\in M$. We begin proving that $(y_{n})$ is bounded. 
Assume by contradiction that, up to a subsequence, $|y_{n}|\rightarrow \infty$ as $n\rightarrow \infty$. \\
Firstly, we consider the case $V_{\infty}= \infty$. Then we can note that $(u_{n})\subset \N_{\e_{n}}$ and Lemma \ref{DI} imply
\begin{align*}
\int_{\R^{3}} V(\e_{n}x+y_{n}) v_{n}^{2} dx &\leq a [v_{n}]^{2} + \int_{\R^{3}} V(\e_{n}x+y_{n}) v_{n}^{2} dx+ \frac{b}{4} [v_{n}]^{4} \\
&\leq \int_{\R^{3}}  f(v_{n}^{2})v_{n}^{2}+ |v_{n}|^{\2} dx, 
\end{align*}
which together with Fatou's Lemma and condition $(V)$ gives 
\begin{align*}
\infty= \liminf_{n\rightarrow \infty} \int_{\R^{3}}  f(v_{n}^{2})v_{n}^{2}+ |v_{n}|^{\2} dx.
\end{align*}
This is impossible because the sequence $(f(v_{n}^{2})v_{n}^{2}+ v_{n}^{\2})$ is bounded in $L^{1}(\R^{3}, \R)$. 

Now we consider the case $V_{\infty}<\infty$. Then, using Lemma \ref{DI}, \eqref{elena} and $V_{0}<V_{\infty}$ we have 
\begin{align*}
m_{V_{0}}&=E_{V_{0}}(\tilde{v})<E_{V_{\infty}}(\tilde{v})= \frac{a}{2}[\tilde{v}]^{2}+\frac{1}{2}\int_{\R^{3}} V_{\infty} \tilde{v}^{2} \, dx+\frac{b}{4}[\tilde{v}]^{4}-\frac{1}{2}\int_{\R^{3}} F(|\tilde{v}|^{2}) dx- \frac{1}{\2} |\tilde{v}|_{\2}^{\2}\\
&\leq \liminf_{n\rightarrow \infty}\left[\frac{a}{2}[\tilde{v}_{n}]^{2}+\frac{1}{2}\int_{\R^{3}} V(\e_{n}x+y_{n}) |\tilde{v}_{n}|^{2} \, dx+\frac{b}{4}[\tilde{v}_{n}]^{4}-\frac{1}{2}\int_{\R^{3}} F(|\tilde{v}_{n}|^{2})\, dx - \frac{1}{\2} |\tilde{v}_{n}|_{\2}^{\2} \right] \\
&\leq \liminf_{n\rightarrow \infty} \left[a\frac{t_{n}^{2}}{2}[|u_{n}|]^{2}+\frac{t_{n}^{2}}{2}\int_{\R^{3}} V(\e_{n}z) |u_{n}|^{2} \, dz
+ b\frac{t_{n}^{4}}{4}[|u_{n}|]^{4}-\frac{1}{2}\int_{\R^{3}} F(|t_{n} u_{n}|^{2}) \, dz -  \frac{t_{n}^{\2}}{\2} |u_{n}|_{\2}^{\2} \right] \\
&\leq \liminf_{n\rightarrow \infty} J_{\e_{n}}(t_{n} u_{n}) \leq \liminf_{n\rightarrow \infty} J_{\e_{n}}(u_{n})= m_{V_{0}}, 
\end{align*}
which is an absurd.

Therefore, $(y_{n})$ is bounded and we may assume that $y_{n}\rightarrow y_{0}\in \R^{3}$. If $y_{0}\notin M$, then $V(y)>V_{0}$ and we get a contradiction arguing as above. Then, $y_{0}\in M$ and this concludes the proof of Lemma. 
\end{proof}

\noindent
Now, we consider the following subset of $\N_{\e}$: 
$$
\widetilde{\N}_{\e}=\left \{u\in \N_{\e}: J_{\e}(u)\leq m_{V_{0}}+h(\e)\right\},
$$
where $h:\R^{+}\rightarrow \R^{+}$ is such that $h(\e)\rightarrow 0$ as $\e \rightarrow 0$.
Fixed $y\in M$, we can see that Lemma \ref{lem3.4} yields $h(\e)=|J_{\e}(\Phi_{\e}(y))-m_{V_{0}}|\rightarrow 0$ as $\e \rightarrow 0$. Therefore $\Phi_{\e}(y)\in \widetilde{\N}_{\e}$, and $\widetilde{\N}_{\e}\neq \emptyset$ for any $\e>0$. 

\begin{lem}\label{lem3.5}
For any $\delta>0$, there holds that
$$
\lim_{\e \rightarrow 0} \sup_{u\in \widetilde{\mathcal{N}}_{\e}} {\rm dist}(\beta_{\e}(u), M_{\delta})=0.
$$
\end{lem}

\begin{proof}
Let $\e_{n}\rightarrow 0$ as $n\rightarrow \infty$. For any $n\in \mathbb{N}$ there exists a sequence $(u_{n})\subset \widetilde{\N}_{\e_{n}}$ such that
\begin{align*}
\sup_{u\in \widetilde{\N}_{\e_{n}}} \inf_{y\in M_{\delta}} |\beta_{\e_{n}} (u)-y| = \inf_{y\in M_{\delta}} |\beta_{\e_{n}}(u_{n}) -y| +o_{n}(1). 
\end{align*}
Therefore, we have to prove that there exists $(y_{n})\subset M_{\delta}$ such that 
\begin{align}\label{4.10AdA}
\lim_{n\rightarrow \infty} |\beta_{\e_{n}}(u_{n})-y_{n} |=0. 
\end{align}
From Lemma \ref{DI} we deduce $E_{V_{0}}(t|u_{n}|) \leq J_{\e_{n}}(tu_{n})$ for any $t\geq 0$. Taking into account $(u_{n})\subset \widetilde{\N}_{\e_{n}}\subset \N_{\e_{n}}$, we get
\begin{align*}
m_{V_{0}}\leq \max_{t\geq 0} E_{V_{0}} (t|u_{n}|) \leq \max_{t\geq 0} J_{\e_{n}}(tu_{n}) = J_{\e_{n}}(u_{n}) \leq m_{V_{0}}+ h(\e_{n}), 
\end{align*}
which together with $h(\e_{n})\rightarrow 0$ as $n\rightarrow \infty$ implies that $J_{\e_{n}}(u_{n})\rightarrow m_{V_{0}}$. Now, by Lemma \ref{prop3.3} there exists $(\tilde{y}_{n})\subset \R^{3}$ such that $y_{n}= \e_{n}\tilde{y}_{n}\in M_{\delta}$ for $n$ sufficiently large. 
Hence, 
\begin{align*}
\beta_{\e_{n}}(u_{n}) = y_{n} + \frac{\int_{\R^{3}} [\Upsilon(\e_{n}z+y_{n}) -y_{n}] |u_{n}(z+\tilde{y}_{n})|^{2}\, dz}{ \int_{\R^{3}} |u_{n}(z+ \tilde{y}_{n})|^{2} dz}. 
\end{align*} 
Recalling that (up to a subsequence) $|u_{n}|(\cdot + \tilde{y}_{n})$ strongly converges in $H^{s}(\R^{3}, \R)$ and using $\e_{n}z+y_{n}\rightarrow y\in M$, we deduce that \eqref{4.10AdA} holds true. 
\end{proof}

\noindent
The following lemma plays a fundamental role in the study the behavior of the maximum points of solutions to \eqref{P}. This result is known in the case $A=0$ (see \cite{AM, A1, DPMV}), but here we have to overcome the presence of the magnetic field $A$. Following \cite{Acpde}, we develop a Moser iteration procedure \cite{Moser} and an approximation argument inspired by Kato's inequality \cite{Kato}.
\begin{lem}\label{moser} 
Let $\e_{n}\rightarrow 0$ and $u_{n}:=u_{\e_{n}}\in \widetilde{\mathcal{N}}_{\e_{n}}$ be a solution to \eqref{Pe}. 
Then $v_{n}=|u_{n}|(\cdot+\tilde{y}_{n})$ satisfies $v_{n}\in L^{\infty}(\R^{3},\R)$ and there exists $C>0$ such that 
$$
|v_{n}|_{\infty}\leq C \mbox{ for all } n\in \mathbb{N},
$$
where $\tilde{y}_{n}$ is given by Lemma \ref{prop3.3}.
Moreover
$$
\lim_{|x|\rightarrow \infty} v_{n}(x)=0 \mbox{ uniformly in } n\in \mathbb{N}.
$$
\end{lem}
\begin{proof}
For any $L>0$, we define $u_{L,n}:=\min\{|u_{n}|, L\}\geq 0$ and we set $v_{L, n}=u_{n} u_{L,n}^{2(\beta-1)}$, where $\beta>1$ will be chosen later.
Taking $v_{L, n}$ as test function in (\ref{Pe}) we can see that
\begin{align}\label{conto1FF}
&(a+b[u_{n}]^{2}_{A_{\e_{n}}})\Re\left(\iint_{\R^{6}} \frac{(u_{n}(x)-u_{n}(y)e^{\imath A_{\e_{n}}(\frac{x+y}{2})\cdot (x-y)})}{|x-y|^{3+2s}} \overline{(u_{n}(x)u_{L,n}^{2(\beta-1)}(x)-u_{n}(y)u_{L,n}^{2(\beta-1)}(y)e^{\imath A_{\e_{n}}(\frac{x+y}{2})\cdot (x-y)})} \, dx dy\right)   \nonumber \\
&=\int_{\R^{3}} f(|u_{n}|^{2}) |u_{n}|^{2}u_{L,n}^{2(\beta-1)}  \,dx-\int_{\R^{3}} V_{\e_{n}} (x) |u_{n}|^{2} u_{L,n}^{2(\beta-1)} \, dx.
\end{align}
Now, we observe that
\begin{align*}
&\Re\left[(u_{n}(x)-u_{n}(y)e^{\imath A_{\e_{n}}(\frac{x+y}{2})\cdot (x-y)})\overline{(u_{n}(x)u_{L,n}^{2(\beta-1)}(x)-u_{n}(y)u_{L,n}^{2(\beta-1)}(y)e^{\imath A_{\e_{n}}(\frac{x+y}{2})\cdot (x-y)})}\right] \\
&=\Re\Bigl[|u_{n}(x)|^{2}u_{L,n}^{2(\beta-1)}(x)-u_{n}(x)\overline{u_{n}(y)} u_{L,n}^{2(\beta-1)}(y)e^{-\imath A_{\e_{n}}(\frac{x+y}{2})\cdot (x-y)}-u_{n}(y)\overline{u_{n}(x)} u_{L,n}^{2(\beta-1)}(x) e^{\imath A_{\e_{n}}(\frac{x+y}{2})\cdot (x-y)} \\
&+|u_{n}(y)|^{2}u_{L,n}^{2(\beta-1)}(y) \Bigr] \\
&\geq (|u_{n}(x)|^{2}u_{L,n}^{2(\beta-1)}(x)-|u_{n}(x)||u_{n}(y)|u_{L,n}^{2(\beta-1)}(y)-|u_{n}(y)||u_{n}(x)|u_{L,n}^{2(\beta-1)}(x)+|u_{n}(y)|^{2}u^{2(\beta-1)}_{L,n}(y) \\
&=(|u_{n}(x)|-|u_{n}(y)|)(|u_{n}(x)|u_{L,n}^{2(\beta-1)}(x)-|u_{n}(y)|u_{L,n}^{2(\beta-1)}(y)).
\end{align*}
Consequently,
\begin{align}\label{realeF}
&\Re\left(\iint_{\R^{6}} \frac{(u_{n}(x)-u_{n}(y)e^{\imath A_{\e_{n}}(\frac{x+y}{2})\cdot (x-y)})}{|x-y|^{3+2s}} \overline{(u_{n}(x)u_{L,n}^{2(\beta-1)}(x)-u_{n}(y)u_{L,n}^{2(\beta-1)}(y)e^{\imath A_{\e_{n}}(\frac{x+y}{2})\cdot (x-y)})} \, dx dy\right) \nonumber\\
&\geq \iint_{\R^{6}} \frac{(|u_{n}(x)|-|u_{n}(y)|)}{|x-y|^{3+2s}} (|u_{n}(x)|u_{L,n}^{2(\beta-1)}(x)-|u_{n}(y)|u_{L,n}^{2(\beta-1)}(y))\, dx dy.
\end{align}
For all $t\geq 0$, we define
\begin{equation*}
\gamma(t):=\gamma_{L, \beta}(t)=t t_{L}^{2(\beta-1)},
\end{equation*}
where  $t_{L}:=\min\{t, L\}$. 
Since $\gamma$ is an increasing function, we have
\begin{align*}
(p-q)(\gamma(p)- \gamma(q))\geq 0 \quad \mbox{ for any } p, q\in \R.
\end{align*}
Let 
\begin{equation*}
\Lambda(t):=\frac{|t|^{2}}{2} \quad \mbox{ and } \quad \Gamma(t):=\int_{0}^{t} (\gamma'(\tau))^{\frac{1}{2}} d\tau. 
\end{equation*}
and we observe
\begin{equation}\label{Gg}
\Lambda'(p-q)(\gamma(p)-\gamma(q))\geq |\Gamma(p)-\Gamma(q)|^{2} \quad \mbox{ for any } p, q\in\R. 
\end{equation}
In fact, for any $p, q\in \R$ such that $p<q$, the Jensen inequality yields
\begin{align*}
\Lambda'(p-q)(\gamma(p)-\gamma(q))&=(p-q)\int_{q}^{p} \gamma'(t)dt\\
&=(p-q)\int_{q}^{p} (\Gamma'(t))^{2}dt \\
&\geq \left(\int_{q}^{p} \Gamma'(t) dt\right)^{2}\\
&=(\Gamma(p)-\Gamma(q))^{2}.
\end{align*}
A similar argument works when $p\geq q$. Therefore, \eqref{Gg} holds true.
By \eqref{Gg}, it follows that
\begin{align}\label{Gg1}
|\Gamma(|u_{n}(x)|)- \Gamma(|u_{n}(y)|)|^{2} \leq (|u_{n}(x)|- |u_{n}(y)|)((|u_{n}|u_{L,n}^{2(\beta-1)})(x)- (|u_{n}|u_{L,n}^{2(\beta-1)})(y)). 
\end{align}
Using \eqref{realeF} and \eqref{Gg1} we obtain
\begin{align}\label{conto1FFF}
&\Re\left(\iint_{\R^{6}} \frac{(u_{n}(x)-u_{n}(y)e^{\imath A_{\e_{n}}(\frac{x+y}{2})\cdot (x-y)})}{|x-y|^{3+2s}} \overline{(u_{n}(x)u_{L,n}^{2(\beta-1)}(x)-u_{n}(y)u_{L,n}^{2(\beta-1)}(y)e^{\imath A_{\e_{n}}(\frac{x+y}{2})\cdot (x-y)})} \, dx dy\right) \nonumber \\
&\geq [\Gamma(|u_{n}|)]^{2}.
\end{align}
Since $\Gamma(|u_{n}|)\geq \frac{1}{\beta} |u_{n}| u_{L,n}^{\beta-1}$, we can use the fractional Sobolev embedding $D^{s,2}(\R^{3}, \R)\subset L^{\2}(\R^{3}, \R)$ to see that
\begin{equation}\label{SS1}
[\Gamma(|u_{n}|)]^{2}\geq S_{*} |\Gamma(|u_{n}|)|^{2}_{\2}\geq \left(\frac{1}{\beta}\right)^{2} S_{*}||u_{n}| u_{L,n}^{\beta-1}|^{2}_{2}.
\end{equation}
Putting together \eqref{conto1FF}, \eqref{conto1FFF}, \eqref{SS1} and noting that $a\leq a+b[u_{n}]_{A_{\e_{n}}}^{2}\leq a+bM^{2}$, we obtain that
\begin{align}\label{BMS}
a\left(\frac{1}{\beta}\right)^{2} S_{*} ||u_{n}| u_{L,n}^{\beta-1}|^{2}_{\2}+\int_{\R^{3}} V_{\e_{n}} (x)|u_{n}|^{2}u_{L,n}^{2(\beta-1)} dx\leq \int_{\R^{3}} f(|u_{n}|^{2}) |u_{n}|^{2} u_{L,n}^{2(\beta-1)} dx.
\end{align}
On the other hand, from $(f_1)$ and $(f_2)$, it follows that for any $\xi>0$ there exists $C_{\xi}>0$ such that
\begin{equation}\label{SS2}
f(t^{2})t^{2}\leq \xi |t|^{2}+C_{\xi}|t|^{\2} \mbox{ for all } t\in \R.
\end{equation}
Taking $\xi\in (0, V_{0})$ and using \eqref{BMS} and \eqref{SS2} we have
\begin{equation}\label{simo1}
|w_{L,n}|_{\2}^{2}\leq C \beta^{2} \int_{\R^{3}} |u_{n}|^{\2}u_{L,n}^{2(\beta-1)},
\end{equation}
where $w_{L,n}:=|u_{n}| u_{L,n}^{\beta-1}$. 
Take $\beta=\frac{\2}{2}$ and fix $R>0$. Recalling that $0\leq u_{L,n}\leq |u_{n}|$ and applying the H\"older inequality, we get
\begin{align}\label{simo2}
\int_{\R^{3}} |u_{n}|^{\2}u_{L,n}^{2(\beta-1)}dx&=\int_{\R^{3}} |u_{n}|^{\2-2} |u_{n}|^{2} u_{L,n}^{\2-2}dx \nonumber\\
&=\int_{\R^{3}} |u_{n}|^{\2-2} (|u_{n}| u_{L,n}^{\frac{\2-2}{2}})^{2}dx \nonumber\\
&\leq \int_{\{|u_{n}|<R\}} R^{\2-2} |u_{n}|^{\2} dx+\int_{\{|u_{n}|>R\}} |u_{n}|^{\2-2} (|u_{n}| u_{L,n}^{\frac{\2-2}{2}})^{2}dx \nonumber\\
&\leq \int_{\{|u_{n}|<R\}} R^{\2-2} |u_{n}|^{\2} dx+\left(\int_{\{|u_{n}|>R\}} |u_{n}|^{\2} dx\right)^{\frac{\2-2}{\2}} \left(\int_{\R^{3}} (|u_{n}| u_{L,n}^{\frac{\2-2}{2}})^{\2}dx\right)^{\frac{2}{\2}}.
\end{align}
Since $(|u_{n}|)$ is bounded in $H^{s}(\R^{3}, \R)$, we can see that for any $R$ sufficiently large
\begin{equation}\label{simo3}
\left(\int_{\{|u_{n}|>R\}} |u_{n}|^{\2} dx\right)^{\frac{\2-2}{\2}}\leq  \frac{1}{2C\beta^{2}}.
\end{equation}
In the light of \eqref{simo1}, \eqref{simo2} and \eqref{simo3}, we infer that
\begin{equation*}
\left(\int_{\R^{3}} (|u_{n}| u_{L,n}^{\frac{\2-2}{2}})^{\2} \right)^{\frac{2}{\2}}\leq C\beta^{2} \int_{\R^{3}} R^{\2-2} |u_{n}|^{\2} dx<\infty,
\end{equation*}
and taking the limit as $L\rightarrow \infty$ we deduce that $|u_{n}|\in L^{\frac{(\2)^{2}}{2}}(\R^{3},\R)$.\\
Since $0\leq u_{L,n}\leq |u_{n}|$ and letting the limit as $L\rightarrow \infty$ in \eqref{simo1}, we get
\begin{equation*}
|u_{n}|_{\beta\2}^{2\beta}\leq C \beta^{2} \int_{\R^{3}} |u_{n}|^{\2+2(\beta-1)},
\end{equation*}
from which we deduce that
\begin{equation*}
\left(\int_{\R^{3}} |u_{n}|^{\beta\2} dx\right)^{\frac{1}{\2(\beta-1)}}\leq (C \beta)^{\frac{1}{\beta-1}} \left(\int_{\R^{3}} |u_{n}|^{\2+2(\beta-1)}\right)^{\frac{1}{2(\beta-1)}}.
\end{equation*}
For $m\geq 1$, we define $\beta_{m+1}$ inductively so that $\2+2(\beta_{m+1}-1)=\2 \beta_{m}$ and $\beta_{1}=\frac{\2}{2}$. \\
Hence
\begin{equation*}
\left(\int_{\R^{3}} |u_{n}|^{\beta_{m+1}\2} dx\right)^{\frac{1}{\2(\beta_{m+1}-1)}}\leq (C \beta_{m+1})^{\frac{1}{\beta_{m+1}-1}} \left(\int_{\R^{3}} |u_{n}|^{\2\beta_{m}}\right)^{\frac{1}{\2(\beta_{m}-1)}}.
\end{equation*}
Setting
$$
D_{m}=\left(\int_{\R^{3}} |u_{n}|^{\2\beta_{m}}\right)^{\frac{1}{\2(\beta_{m}-1)}},
$$
we can use an iterative argument to deduce that there is $C_{0}>0$ independent of $m$ such that 
$$
D_{m+1}\leq \prod_{k=1}^{m} (C \beta_{k+1})^{\frac{1}{\beta_{k+1}-1}}  D_{1}\leq C_{0} D_{1}.
$$
Taking the limit as $m\rightarrow \infty$ we get 
\begin{equation}\label{UBu}
|u_{n}|_{\infty}\leq C_{0}D_{1}=:K \mbox{ for all } n\in \mathbb{N}.
\end{equation}
Moreover, by interpolation, we can infer that $(|u_{n}|)$ strongly converges in $L^{r}(\R^{3}, \R)$ for all $r\in (2, \infty)$. By $(f_1)$-$(f_2)$, we can see that $f(|u_{n}|^{2})|u_{n}|$ strongly converges in $L^{r}(\R^{3}, \R)$ for all $r\in [2, \infty)$.

Now, we prove that $v_{n}$ vanishes at infinity uniformly in $n\in \mathbb{N}$. Firstly, we show that $|u_{n}|$ is a weak subsolution to 
\begin{equation}\label{Kato0}
\left\{
\begin{array}{ll}
(a+b[v]^{2})(-\Delta)^{s}v+V_{0} v=f(v^{2})v+v^{\2-1} &\mbox{ in } \R^{3}, \\
v\geq 0 \quad \mbox{ in } \R^{3}.
\end{array}
\right.
\end{equation}
To do this, we proceed as in \cite{Acpde} using a sort of approximated Kato's inequality for solutions of \eqref{Pe}.

Take $\varphi\in C^{\infty}_{c}(\R^{3}, \R)$ such that $\varphi\geq 0$, and we use $\psi_{\delta, n}=\frac{u_{n}}{u_{\delta, n}}\varphi$ as test function in \eqref{Pe}, where $u_{\delta,n}=\sqrt{|u_{n}|^{2}+\delta^{2}}$ for $\delta>0$. Note that $\psi_{\delta, n}\in H^{s}_{\e_{n}}$ for all $\delta>0$ and $n\in \mathbb{N}$. Indeed $\int_{\R^{3}} V_{\e_{n}} (x) |\psi_{\delta,n}|^{2} dx\leq \int_{\supp(\varphi)} V_{\e_{n}} (x)\varphi^{2} dx<\infty$. 
Now, we note that 
\begin{align*}
\psi_{\delta,n}(x)-\psi_{\delta,n}(y)e^{\imath A_{\e_{n}}(\frac{x+y}{2})\cdot (x-y)}&=\left(\frac{u_{n}(x)}{u_{\delta,n}(x)}\right)\varphi(x)-\left(\frac{u_{n}(y)}{u_{\delta,n}(y)}\right)\varphi(y)e^{\imath A_{\e_{n}}(\frac{x+y}{2})\cdot (x-y)}\\
&=\left[\left(\frac{u_{n}(x)}{u_{\delta,n}(x)}\right)-\left(\frac{u_{n}(y)}{u_{\delta,n}(x)}\right)e^{\imath A_{\e_{n}}(\frac{x+y}{2})\cdot (x-y)}\right]\varphi(x) \\
&+\left[\varphi(x)-\varphi(y)\right] \left(\frac{u_{n}(y)}{u_{\delta,n}(x)}\right) e^{\imath A_{\e_{n}}(\frac{x+y}{2})\cdot (x-y)} \\
&+\left(\frac{u_{n}(y)}{u_{\delta,n}(x)}-\frac{u_{n}(y)}{u_{\delta,n}(y)}\right)\varphi(y) e^{\imath A_{\e_{n}}(\frac{x+y}{2})\cdot (x-y)}.
\end{align*}
Putting together $|z+w+k|^{2}\leq 4(|z|^{2}+|w|^{2}+|k|^{2})$ for all $z,w,k\in \C$, $|e^{\imath t}|=1$ for all $t\in \R$, $u_{\delta,n}\geq \delta$, $|\frac{u_{n}}{u_{\delta,n}}|\leq 1$, \eqref{UBu} and $|\sqrt{|z|^{2}+\delta^{2}}-\sqrt{|w|^{2}+\delta^{2}}|\leq ||z|-|w||$ for all $z, w\in \C$, we can deduce that
\begin{align*}
&|\psi_{\delta,n}(x)-\psi_{\delta,n}(y)e^{\imath A_{\e_{n}}(\frac{x+y}{2})\cdot (x-y)}|^{2} \\
&\leq \frac{4}{\delta^{2}}|u_{n}(x)-u_{n}(y)e^{\imath A_{\e_{n}}(\frac{x+y}{2})\cdot (x-y)}|^{2}|\varphi|^{2}_{\infty} +\frac{4}{\delta^{2}}|\varphi(x)-\varphi(y)|^{2} |u_{n}|^{2}_{\infty} \\
&+\frac{4}{\delta^{4}} |u_{n}|^{2}_{\infty} \|\varphi\|^{2}_{L^{\infty}(\R^{3})} |u_{\delta,n}(y)-u_{\delta,n}(x)|^{2} \\
&\leq \frac{4}{\delta^{2}}|u_{n}(x)-u_{n}(y)e^{\imath A_{\e_{n}}(\frac{x+y}{2})\cdot (x-y)}|^{2}|\varphi|^{2}_{\infty} +\frac{4K^{2}}{\delta^{2}}|\varphi(x)-\varphi(y)|^{2} \\
&+\frac{4K^{2}}{\delta^{4}} |\varphi|^{2}_{\infty} ||u_{n}(y)|-|u_{n}(x)||^{2}. 
\end{align*}
In view of $u_{n}\in H^{s}_{\e_{n}}$, $|u_{n}|\in H^{s}(\R^{3}, \R)$ (by Lemma \ref{DI}) and $\varphi\in C^{\infty}_{c}(\R^{3}, \R)$, we get $\psi_{\delta,n}\in H^{s}_{\e_{n}}$.
Consequently,
\begin{align}\label{Kato1}
&(a+b[u_{n}]^{2}_{A_{\e_{n}}})\Re\left[\iint_{\R^{6}} \frac{(u_{n}(x)-u_{n}(y)e^{\imath A_{\e_{n}}(\frac{x+y}{2})\cdot (x-y)})}{|x-y|^{3+2s}} \left(\frac{\overline{u_{n}(x)}}{u_{\delta,n}(x)}\varphi(x)-\frac{\overline{u_{n}(y)}}{u_{\delta,n}(y)}\varphi(y)e^{-\imath A_{\e_{n}}(\frac{x+y}{2})\cdot (x-y)}  \right) dx dy\right] \nonumber\\
&+\int_{\R^{3}} V_{\e_{n}} (x)\frac{|u_{n}|^{2}}{u_{\delta,n}}\varphi dx=\int_{\R^{3}} f(|u_{n}|^{2})\frac{|u_{n}|^{2}}{u_{\delta,n}}\varphi+\frac{|u_{n}|^{\2}}{u_{\delta, n}} \varphi dx.
\end{align}
Since $\Re(z)\leq |z|$ for all $z\in \C$ and  $|e^{\imath t}|=1$ for all $t\in \R$, it follows that
\begin{align}\label{alves1}
&\Re\left[(u_{n}(x)-u_{n}(y)e^{\imath A_{\e_{n}}(\frac{x+y}{2})\cdot (x-y)}) \left(\frac{\overline{u_{n}(x)}}{u_{\delta,n}(x)}\varphi(x)-\frac{\overline{u_{n}(y)}}{u_{\delta,n}(y)}\varphi(y)e^{-\imath A_{\e_{n}}(\frac{x+y}{2})\cdot (x-y)}  \right)\right] \nonumber\\
&=\Re\left[\frac{|u_{n}(x)|^{2}}{u_{\delta,n}(x)}\varphi(x)+\frac{|u_{n}(y)|^{2}}{u_{\delta,n}(y)}\varphi(y)-\frac{u_{n}(x)\overline{u_{n}(y)}}{u_{\delta,n}(y)}\varphi(y)e^{-\imath A_{\e_{n}}(\frac{x+y}{2})\cdot (x-y)} -\frac{u_{n}(y)\overline{u_{n}(x)}}{u_{\delta,n}(x)}\varphi(x)e^{\imath A_{\e_{n}}(\frac{x+y}{2})\cdot (x-y)}\right] \nonumber \\
&\geq \left[\frac{|u_{n}(x)|^{2}}{u_{\delta,n}(x)}\varphi(x)+\frac{|u_{n}(y)|^{2}}{u_{\delta,n}(y)}\varphi(y)-|u_{n}(x)|\frac{|u_{n}(y)|}{u_{\delta,n}(y)}\varphi(y)-|u_{n}(y)|\frac{|u_{n}(x)|}{u_{\delta,n}(x)}\varphi(x) \right].
\end{align}
Now, we note that
\begin{align}\label{alves2}
&\frac{|u_{n}(x)|^{2}}{u_{\delta,n}(x)}\varphi(x)+\frac{|u_{n}(y)|^{2}}{u_{\delta,n}(y)}\varphi(y)-|u_{n}(x)|\frac{|u_{n}(y)|}{u_{\delta,n}(y)}\varphi(y)-|u_{n}(y)|\frac{|u_{n}(x)|}{u_{\delta,n}(x)}\varphi(x) \nonumber\\
&=  \frac{|u_{n}(x)|}{u_{\delta,n}(x)}(|u_{n}(x)|-|u_{n}(y)|)\varphi(x)-\frac{|u_{n}(y)|}{u_{\delta,n}(y)}(|u_{n}(x)|-|u_{n}(y)|)\varphi(y) \nonumber\\
&=\left[\frac{|u_{n}(x)|}{u_{\delta,n}(x)}(|u_{n}(x)|-|u_{n}(y)|)\varphi(x)-\frac{|u_{n}(x)|}{u_{\delta,n}(x)}(|u_{n}(x)|-|u_{n}(y)|)\varphi(y)\right] \nonumber\\
&+\left(\frac{|u_{n}(x)|}{u_{\delta,n}(x)}-\frac{|u_{n}(y)|}{u_{\delta,n}(y)} \right) (|u_{n}(x)|-|u_{n}(y)|)\varphi(y) \nonumber\\
&=\frac{|u_{n}(x)|}{u_{\delta,n}(x)}(|u_{n}(x)|-|u_{n}(y)|)(\varphi(x)-\varphi(y)) +\left(\frac{|u_{n}(x)|}{u_{\delta,n}(x)}-\frac{|u_{n}(y)|}{u_{\delta,n}(y)} \right) (|u_{n}(x)|-|u_{n}(y)|)\varphi(y) \nonumber\\
&\geq \frac{|u_{n}(x)|}{u_{\delta,n}(x)}(|u_{n}(x)|-|u_{n}(y)|)(\varphi(x)-\varphi(y)) 
\end{align}
where in the last inequality we used that
$$
\left(\frac{|u_{n}(x)|}{u_{\delta,n}(x)}-\frac{|u_{n}(y)|}{u_{\delta,n}(y)} \right) (|u_{n}(x)|-|u_{n}(y)|)\varphi(y)\geq 0
$$
being
$$
h(t)=\frac{t}{\sqrt{t^{2}+\delta^{2}}} \mbox{ is increasing for } t\geq 0 \quad \mbox{ and } \quad \varphi\geq 0 \mbox{ in }\R^{3}.
$$
Observing that
$$
\frac{|\frac{|u_{n}(x)|}{u_{\delta,n}(x)}(|u_{n}(x)|-|u_{n}(y)|)(\varphi(x)-\varphi(y))|}{|x-y|^{3+2s}}\leq \frac{||u_{n}(x)|-|u_{n}(y)||}{|x-y|^{\frac{3+2s}{2}}} \frac{|\varphi(x)-\varphi(y)|}{|x-y|^{\frac{3+2s}{2}}}\in L^{1}(\R^{6}),
$$
and $\frac{|u_{n}(x)|}{u_{\delta,n}(x)}\rightarrow 1$ a.e. in $\R^{3}$ as $\delta\rightarrow 0$,
we can apply \eqref{alves1}, \eqref{alves2} and the Dominated Convergence Theorem to deduce that
\begin{align}\label{Kato2}
&\limsup_{\delta\rightarrow 0} \Re\left[\iint_{\R^{6}} \frac{(u_{n}(x)-u_{n}(y)e^{\imath A_{\e_{n}}(\frac{x+y}{2})\cdot (x-y)})}{|x-y|^{3+2s}} \left(\frac{\overline{u_{n}(x)}}{u_{\delta,n}(x)}\varphi(x)-\frac{\overline{u_{n}(y)}}{u_{\delta,n}(y)}\varphi(y)e^{-\imath A_{\e_{n}}(\frac{x+y}{2})\cdot (x-y)}  \right) dx dy\right] \nonumber\\
&\geq \limsup_{\delta\rightarrow 0} \iint_{\R^{6}} \frac{|u_{n}(x)|}{u_{\delta,n}(x)}(|u_{n}(x)|-|u_{n}(y)|)(\varphi(x)-\varphi(y)) \frac{dx dy}{|x-y|^{3+2s}} \nonumber\\
&=\iint_{\R^{6}} \frac{(|u_{n}(x)|-|u_{n}(y)|)(\varphi(x)-\varphi(y))}{|x-y|^{3+2s}} dx dy.
\end{align}
On the other hand, from the Dominated Convergence Theorem (we recall that $\frac{|u_{n}|^{2}}{u_{\delta, n}}\leq |u_{n}|$ and $\varphi\in C^{\infty}_{c}(\R^{3}, \R)$)  it follows that
\begin{equation}\label{Kato3}
\lim_{\delta\rightarrow 0} \int_{\R^{3}} V_{\e_{n}} (x)\frac{|u_{n}|^{2}}{u_{\delta,n}}\varphi dx=\int_{\R^{3}} V_{\e_{n}} (x)|u_{n}|\varphi dx\geq \int_{\R^{3}} V_{0}|u_{n}|\varphi dx,
\end{equation}
\begin{equation}\label{KatoP}
\liminf_{\delta\rightarrow 0} \int_{\R^{3}} \frac{|u_{n}|^{\2}}{u_{\delta, n}}\varphi dx= \int_{\R^{3}}  |u|^{\2-1}\varphi dx,
\end{equation}
and
\begin{equation}\label{Kato4}
\lim_{\delta\rightarrow 0}  \int_{\R^{3}} f(|u_{n}|^{2})\frac{|u_{n}|^{2}}{u_{\delta,n}}\varphi dx=\int_{\R^{3}} f(|u_{n}|^{2}) |u_{n}|\varphi dx.
\end{equation}
By Lemma \ref{DI} we can also see that
\begin{equation}\label{Kff}
a+b[u_{n}]^{2}_{A_{\e_{n}}}\geq a+b[|u_{n}|]^{2}.
\end{equation}
Putting together \eqref{Kato1}, \eqref{Kato2}, \eqref{Kato3}, \eqref{KatoP}, \eqref{Kato4} and \eqref{Kff} we can infer that
\begin{align*}
(a+b[|u_{n}|]^{2})\iint_{\R^{6}} \frac{(|u_{n}(x)|-|u_{n}(y)|)(\varphi(x)-\varphi(y))}{|x-y|^{3+2s}}& dx dy+\int_{\R^{3}} V_{0}|u_{n}|\varphi dx \\
&\leq \int_{\R^{3}} [f(|u_{n}|^{2}) |u_{n}|+|u_{n}|^{\2-1}]\varphi dx
\end{align*}
for any $\varphi\in C^{\infty}_{c}(\R^{3}, \R)$ such that $\varphi\geq 0$, that is $|u_{n}|$ is a weak subsolution to \eqref{Kato0}.\\
Now, we define $v_{n}=|u_{n}|(\cdot+\tilde{y}_{n})$. Then Lemma \ref{DI} yields
$$
a+b[v_{n}]^{2}=a+b[|u_{n}|]^{2}\leq a+b[u_{n}]_{A_{\e_{n}}}^{2}\leq a+bM^{2}.
$$
Note that $v_{n}$ verifies 
\begin{equation}\label{Pkat}
(-\Delta)^{s} v_{n} + \frac{V_{0}}{a+bM^{2}} v_{n}\leq f_{n} \mbox{ in } \R^{3},
\end{equation}
where
$$
f_{n}:=(a+b[v_{n}]^{2})^{-1}[f(v_{n}^{2})v_{n}+v_{n}^{\2-1}-V_{0}v_{n}]+\frac{V_{0}}{a+bM^{2}}v_{n}.
$$
Let $z_{n}\in H^{s}(\R^{3}, \R)$ be the unique solution to
\begin{equation}\label{US}
(-\Delta)^{s} z_{n} + \frac{V_{0}}{a+bM^{2}} z_{n}=f_{n} \quad \mbox{ in } \R^{3}.
\end{equation}
In the light of \eqref{UBu}, we know that $|v_{n}|_{\infty}\leq C$ for all $n\in \mathbb{N}$, and by interpolation $v_{n}\rightarrow v$ strongly converges in $L^{r}(\R^{3}, \R)$ for all $r\in (2, \infty)$, for some $v\in L^{r}(\R^{3}, \R)$.
From the growth assumptions on $f$, we can see that 
$$
f_{n}\rightarrow  (a+b[v]^{2})^{-1}[f(v^{2})v+v^{\2-1}-V_{0}v]+\frac{V_{0}}{a+bM^{2}}v \quad \mbox{ in } L^{r}(\R^{3}, \R) \quad \forall r\in [2, \infty),
$$ 
and there exists $C>0$ such that $|f_{n}|_{\infty}\leq C$ for all $n\in \mathbb{N}$.
Then $z_{n}=\mathcal{K}*g_{n}$ (see \cite{FQT}), where $\mathcal{K}$ is the Bessel kernel, and arguing as in \cite{AM}, we can see that $|z_{n}(x)|\rightarrow 0$ as $|x|\rightarrow \infty$ uniformly in $n\in \mathbb{N}$.
Taking into account $v_{n}$ satisfies \eqref{Pkat} and $z_{n}$ solves \eqref{US}, we can see that by comparison, $0\leq v_{n}\leq z_{n}$ a.e. in $\R^{3}$ and for all $n\in \mathbb{N}$. Therefore, $v_{n}(x)\rightarrow 0$ as $|x|\rightarrow \infty$ uniformly in $n\in \mathbb{N}$.
\end{proof}

\begin{remark}
We can also prove that $v_{n}(x)\rightarrow 0$ as $|x|\rightarrow \infty$ uniformly in $n\in \mathbb{N}$ following the approach in \cite{Amjm}, that is showing that $|u_{n}|^{2}$ is a sub-solution of a suitable fractional Kirchhoff equation. Anyway, the argument
used in Lemma \ref{moser}  will be useful to deduce informations on the decay estimate of solutions to \eqref{P}; see the last part of the proof of Theorem \ref{thm1} below.
\end{remark}

\noindent
We conclude this section giving the proof of the main result of this paper.
\begin{proof}[Proof of Theorem \ref{thm1}] \hfill\\
{\bf Multiplicity of solutions} \\
For any fixed  $\delta>0$, we can invoke Lemma \ref{lem3.4}, Lemma \ref{lem3.5N}, Lemma \ref{lem3.5} and argue as in \cite{CL} to obtain some $\e_{\delta}>0$ such that, for any $\e\in (0, \e_{\delta})$, the following diagram
$
M \stackrel{\Phi_{\e}}\rightarrow \widetilde{\mathcal{N}}_{\e} \stackrel{\beta_{\e}}\rightarrow M_{\delta}
$
is well defined. In particular, for $\e>0$ small enough, we have $\beta_{\e}(\Phi_{\e}(y))=y+\alpha(y)$ for $y\in M$, where $|\alpha(y)|<\frac{\delta}{2}$ uniformly for $y\in M$.
Let us define $H: [0, 1]\times M\rightarrow M_{\delta}$ by $H(t, y)=y+(1-t)\alpha(y)$. Thus, it is easy to check that $H$ is continuous and $H(0, y)=y$ for all $y\in M$, which implies that $\beta_{\e}\circ \Phi_{\e}$ is homotopically equivalent to the embedding  $\iota: M\rightarrow M_{\delta}$.
Arguing as in \cite{BC}, we can see that $cat_{\widetilde{\mathcal{N}}_{\e}}(\widetilde{\mathcal{N}}_{\e})\geq cat_{M_{\delta}}(M)$. On the other hand, using the definition of $\widetilde{\mathcal{N}}_{\e}$ and choosing $\e_{\delta}$ sufficiently small, we can use Proposition \ref{propPSc} to deduce that $J_{\e}$ verifies the Palais-Smale condition in $\widetilde{\mathcal{N}}_{\e}$.
Applying standard Ljusternik-Schnirelmann theory \cite{W}, we can infer that $J_{\e}$  possesses at least $cat_{\widetilde{\mathcal{N}}_{\e}}(\widetilde{\mathcal{N}}_{\e})$  critical points on $\mathcal{N}_{\e}$. In view of Corollary \ref{cor},  we obtain at least $cat_{M_{\delta}}(M)$ nontrivial solutions for  \eqref{Pe}.

\noindent
{\bf Concentration of the maximum points} \\
Take $\e_{n}\rightarrow 0$ and $u_{n}:=u_{\e_{n}}$ be a solutions to \eqref{Pe} and we set $v_{n}(x)= |u_{n}|(x+\tilde{y}_{n})$. Let us prove that $|v_{n}|_{\infty}\geq \rho$ for all $n\in \mathbb{N}$. Assume by contradiction that $|v_{n}|_{\infty} \rightarrow 0$. Using $(f_{1})$ we can find $n_{0}\in \mathbb{N}$ such that 
\begin{align}\label{4.18HZ}
\frac{f(|v_{n}|_{\infty}^{2})}{|v_{n}|_{\infty}^{2}}< \frac{V_{0}}{2} \quad \mbox{ for all } n\geq n_{0}.
\end{align}
Then, exploiting $J_{\e_{n}}'(u_{n})=0$, \eqref{4.18HZ}, Lemma \ref{DI} and $(f_{4})$ we can deduce that 
\begin{align*}
a[v_{n}]^{2} + V_{0} |v_{n}|_{2}^{2} &= a[|u_{n}|]^{2}+\int_{\R^{3}}V_{0}|u_{n}|^{2}dx\\
&\leq \|u_{n}\|^{2}_{\e_{n}}+b[u_{n}]^{4}_{A_{\e_{n}}}\\
&= \int_{\R^{3}} f(|u_{n}|^{2}) |u_{n}|^{2}\, dx + |u_{n}|_{\2}^{\2}\\
&= \int_{\R^{3}} f(v_{n}^{2}) v_{n}^{2} \, dx + |v_{n}|_{\2}^{\2}\\
&\leq \int_{\R^{3}} \frac{f(|v_{n}|_{\infty}^{2})}{|v_{n}|_{\infty}^{2}} |v_{n}|^{4} \, dx + |v_{n}|_{\2}^{\2}\\
&\leq \frac{V_{0}}{2} |v_{n}|_{\infty}^{2} |v_{n}|_{2}^{2} + |v_{n}|_{\infty}^{\2-2} |v_{n}|_{2}^{2}
\end{align*}
and this implies that $\|v_{n}\|_{V_{0}}=0$ for all $n\geq n_{0}$. This gives a contradiction because $v_{n}\rightarrow v$ in $H^{s}_{V_{0}}$ and $v\neq 0$ by Lemma \ref{prop3.3}.  
Therefore, we deduce that $\rho \leq |v_{n}|_{\infty}\leq C$ for all $n\in \mathbb{N}$. Hence, we can see that the maximum points $p_{n}$ of $|u_{n}|$ belong to $B_{R}(\tilde{y}_{n})$, that is $p_{n}=\tilde{y}_{n}+q_{n}$ for some $q_{n}\in B_{R}(0)$.  Since the associated solution of \eqref{P} is of the form $\hat{u}_{n}(x)=u_{n}(x/\e_{n})$, we can infer that a maximum point $\eta_{\e_{n}}$ of $|\hat{u}_{n}|$ is $\eta_{\e_{n}}=\e_{n}\tilde{y}_{n}+\e_{n}q_{n}$. In view of $q_{n}\in B_{R}(0)$, $\e_{n}\tilde{y}_{n}\rightarrow y_{0}$ and $V(y_{0})=V_{0}$, we can use the continuity of $V$ to deduce that
$$
\lim_{n\rightarrow \infty} V(\eta_{\e_{n}})=V_{0}.
$$
{\bf Decay estimate} \\
Invoking Lemma $4.3$ in \cite{FQT}, we can find a function $w$ and a constant $R>0$ such that 
\begin{align}\label{HZ1}
0<w(x)\leq \frac{C}{1+|x|^{3+2s}},
\end{align}
and
\begin{align}\label{HZ2}
(-\Delta)^{s} w+\frac{V_{0}}{2(a+bM^{2})}w\geq 0 \mbox{ in } B^{c}_{R_{1}}(0) 
\end{align}
where $M>0$ is such that $a+bM^{2}\geq a+b[u_{n}]^{2}_{A_{\e_{n}}}\geq a+b[v_{n}]^{2}$ (the last inequality is due to Lemma \ref{DI}). Since Lemma \ref{moser} implies that $v_{n}(x)\rightarrow 0$ as $|x|\rightarrow \infty$ uniformly in $n\in \mathbb{N}$, by $(f_1)$ it follows that there exists $R_{2}>0$ such that
\begin{equation}\label{hzero}
f(v_{n}^{2})v_{n}+v_{n}^{\2-1}\leq \frac{V_{0}}{2}v_{n}  \mbox{ in } B_{R_{2}}^{c}(0).
\end{equation}
Arguing as in Lemma \ref{moser}, we can see that $v_{n}$ verifies
\begin{equation}\label{Pkat}
(-\Delta)^{s} v_{n} + \frac{V_{0}}{a+bM^{2}} v_{n}\leq f_{n} \mbox{ in } \R^{3},
\end{equation}
where
$$
f_{n}:=(a+b[v_{n}]^{2})^{-1}[f(v_{n}^{2})v_{n}+v_{n}^{\2-1}-V(\e_{n} x+\e_{n}\tilde{y}_{n})v_{n}]+\frac{V_{0}}{a+bM^{2}}v_{n}.
$$
Let $w_{n}$ be the unique solution to 
$$
(-\Delta)^{s}w_{n}+\frac{V_{0}}{(a+bM^{2})}w_{n}=f_{n} \mbox{ in } \R^{3}.
$$
Then, by comparison, $0\leq v_{n}\leq w_{n}$ in $\R^{3}$ and using \eqref{hzero} we obtain
\begin{align*}
(-\Delta)^{s}w_{n}&+\frac{V_{0}}{2(a+bM^{2})}w_{n} \\
&=(-\Delta)^{s}w_{n}+\frac{V_{0}}{(a+bM^{2})}w_{n}-\frac{V_{0}}{2(a+bM^{2})}w_{n} \\
&\leq f_{n}-\frac{V_{0}}{2(a+bM^{2})}v_{n} \\
&\leq (a+b[v_{n}]^{2})^{-1}[f(v_{n}^{2})v_{n}+v_{n}^{\2-1}-V(\e_{n} x+\e_{n}\tilde{y}_{n})v_{n}]+\frac{V_{0}}{2(a+bM^{2})}v_{n} \\
&\leq (a+b[v_{n}]^{2})^{-1}\left\{f(v_{n}^{2})v_{n}+v_{n}^{\2-1}-\left(V(\e_{n} x+\e_{n}\tilde{y}_{n})-\frac{V_{0}}{2}\right)v_{n}\right\} \\
&\leq (a+b[v_{n}]^{2})^{-1}\left\{f(v_{n}^{2})v_{n}+v_{n}^{\2-1}-\frac{V_{0}}{2} v_{n}\right\} \leq 0 \mbox{ in } B_{R_{2}}^{c}(0).
\end{align*}
Set $R_{3}:=\max\{R_{1}, R_{2}\}$ and we define
\begin{align}\label{HZ4}
c:=\inf_{B_{R_{3}}(0)} w>0 \mbox{ and } \tilde{w}_{n}:=(d+1)w-c w_{n},
\end{align}
where $d:=\sup_{n\in \mathbb{N}} |w_{n}|_{\infty}<\infty$. 
Our claim is to prove that
\begin{equation}\label{HZ5}
\tilde{w}_{n}\geq 0 \mbox{ in } \R^{3}.
\end{equation}
We begin observing that
\begin{align}
&\lim_{|x|\rightarrow \infty} \sup_{n\in \mathbb{N}}\tilde{w}_{n}(x)=0,  \label{HZ0N} \\
&\tilde{w}_{n}\geq dc+w-dc>0 \mbox{ in } B_{R_{3}}(0) \label{HZ0},\\
&(-\Delta)^{s} \tilde{w}_{n}+\frac{V_{0}}{2(a+bM^{2})}\tilde{w}_{n}\geq 0 \mbox{ in } B^{c}_{R_{3}}(0) \label{HZ00}.
\end{align}
Now, we assume by contradiction that there exists a sequence $(\bar{x}_{j, n})\subset \R^{3}$ such that 
\begin{align}\label{HZ6}
\inf_{x\in \R^{3}} \tilde{w}_{n}(x)=\lim_{j\rightarrow \infty} \tilde{w}_{n}(\bar{x}_{j, n})<0. 
\end{align}
Then, by  (\ref{HZ0N}), $(\bar{x}_{j, n})$ is bounded and, up to subsequence, we may assume that there exists $\bar{x}_{n}\in \R^{3}$ such that $\bar{x}_{j, n}\rightarrow \bar{x}_{n}$ as $j\rightarrow \infty$. 
Consequently, (\ref{HZ6}) implies that
\begin{align}\label{HZ7}
\inf_{x\in \R^{3}} \tilde{w}_{n}(x)= \tilde{w}_{n}(\bar{x}_{n})<0.
\end{align}
Using the minimality of $\bar{x}_{n}$ and the representation formula for the fractional Laplacian (see Lemma $3.2$ in \cite{DPV}), we can deduce that 
\begin{align}\label{HZ8}
(-\Delta)^{s}\tilde{w}_{n}(\bar{x}_{n})=\frac{c_{3, s}}{2} \int_{\R^{3}} \frac{2\tilde{w}_{n}(\bar{x}_{n})-\tilde{w}_{n}(\bar{x}_{n}+\xi)-\tilde{w}_{n}(\bar{x}_{n}-\xi)}{|\xi|^{3+2s}} d\xi\leq 0.
\end{align}
Combining (\ref{HZ0}) and (\ref{HZ6}), we obtain $\bar{x}_{n}\in  B^{c}_{R_{3}}(0)$,
which together with (\ref{HZ7}) and (\ref{HZ8}) gives
$$
(-\Delta)^{s} \tilde{w}_{n}(\bar{x}_{n})+\frac{V_{0}}{2(a+bM^{2})}\tilde{w}_{n}(\bar{x}_{n})<0.
$$
This contradicts (\ref{HZ00}), and then (\ref{HZ5}) holds true.
Taking into account (\ref{HZ1}) and $v_{n}\leq w_{n}$ we can see that
\begin{align*}
0\leq v_{n}(x)\leq w_{n}(x)\leq \frac{\tilde{C}}{1+|x|^{3+2s}} \mbox{ for all } n\in \mathbb{N}, x\in \R^{3},
\end{align*}
for some constant $\tilde{C}>0$. 
Therefore, recalling the definition of $v_{n}$, we can deduce that
\begin{align*}
|\hat{u}_{n}|(x)&=|u_{n}|\left(\frac{x}{\e_{n}}\right)=v_{n}\left(\frac{x}{\e_{n}}-\tilde{y}_{n}\right) \\
&\leq \frac{\tilde{C}}{1+|\frac{x}{\e_{n}}-\tilde{y}_{n}|^{3+2s}} \\
&=\frac{\tilde{C} \e_{n}^{3+2s}}{\e_{n}^{3+2s}+|x- \e_{n} \tilde{y}_{n}|^{3+2s}} \\
&\leq \frac{\tilde{C} \e_{n}^{3+2s}}{\e_{n}^{3+2s}+|x-\eta_{\e_{n}}|^{3+2s}}.
\end{align*}
\end{proof}
This ends the proof of Theorem \ref{thm1}.

\noindent
{\bf Acknowledgments.}
The author would like to express his sincere gratitude to the referee for all insightful comments and valuable suggestions, which enabled to improve this version of the manuscript.

\end{document}